\documentclass[12pt,oneside,reqno]{amsart}

\usepackage{amsfonts}
\usepackage{amsmath}
\usepackage{amssymb}
\usepackage{amsthm}
\usepackage{enumerate}
\usepackage{graphicx}
\usepackage{hyperref}
\usepackage{mathrsfs}
\usepackage{stmaryrd}
\usepackage{color}

\newcommand{\be}{\begin{eqnarray}}
\newcommand{\ee}{\end{eqnarray}}
\newcommand{\ce}{\begin{eqnarray*}}
\newcommand{\de}{\end{eqnarray*}}

\newtheorem{theorem}{Theorem}[section]
\newtheorem{proposition}[theorem]{Proposition}
\newtheorem{lemma}[theorem]{Lemma}
\newtheorem{corollary}[theorem]{Corollary}
\newtheorem{remark}[theorem]{Remark}
\newtheorem{definition}[theorem]{Definition}
\newtheorem{examples}[theorem]{Examples}
\newtheorem{assumption}[theorem]{Assumption}

\def\bt{\begin{theorem}}
\def\et{\end{theorem}}
\def\bp{\begin{proposition}}
\def\ep{\end{proposition}}
\def\bl{\begin{lemma}}
\def\el{\end{lemma}}
\def\bc{\begin{corollary}}
\def\ec{\end{corollary}}
\def\bd{\begin{definition}}
\def\ed{\end{definition}}
\def\br{\begin{remark}}
\def\er{\end{remark}}
\def\bx{\begin{examples}}
\def\ex{\end{examples}}
\def\ba{\begin{assumption}}
\def\ea{\end{assumption}}

\def\e{\varepsilon}

\def\[{{\Big[}}
\def\]{{\Big]}}
\def\<{{\langle}}
\def\>{{\rangle}}
\def\({{\Big(}}
\def\){{\Big)}}

\def\geq{\geqslant}
\def\leq{\leqslant}

\def\no{\nonumber}

\def\min{{\mathord{{\rm min}}}}

\def\cA{{\mathcal A}}
\def\cB{{\mathcal B}}
\def\cC{{\mathcal C}}
\def\cD{{\mathcal D}}
\def\cE{{\mathcal E}}

\def\cG{{\mathcal G}}

\def\cO{{\mathcal O}}

\def\cR{{\mathcal R}}

\def\cX{{\mathcal X}}

\def\mD{{\mathbb D}}
\def\mE{{\mathbb E}}

\def\mH{{\mathbb H}}

\def\mN{{\mathbb N}}

\def\mP{{\mathbb P}}

\def\mR{{\mathbb R}}

\allowdisplaybreaks

\begin{document}

\title{Derivative Formulae and Gradient Estimates for Stochastic Hamiltonian Systems with Jumps}

\author{Hua Zhang}

\subjclass[2010]{}

\date{}

\dedicatory{School of Statistics and Data Science, Jiangxi University of Finance and Economics,\\
	Nanchang, Jiangxi 330013, P.R.China\\
	Emails: H. Zhang: zh860801@163.com}


\thanks{This work is supported by National Natural Science Foundation of China (Grant Nos. 12261038 and 12371152), and Natural Science Foundation of Jiangxi Province (Grant Nos. 20232BAB201004 and 20242BAB23003).}

\begin{abstract}
This paper is concerned with stochastic Hamiltonian systems driven by L\'{e}vy processes. By employing the lent particle method developed by Bouleau and Denis in the framework of Malliavin calculus with jumps, we establish an explicit Bismut--Elworthy--Li type derivative formula for the associated Markov semigroup, as well as corresponding gradient estimates. The main novelty lies in extending the known results for Brownian-motion-driven degenerate systems to the nonlocal setting with jumps, where the presence of a pure jump noise and the degeneracy of the coefficients pose substantial difficulties. Our approach provides a systematic treatment of such nonlocal degenerate operators and fills a gap in the literature on derivative formulae for jump-type stochastic Hamiltonian systems.
\end{abstract}

\maketitle

\section{Introduction}\label{sec:intro}

In the present paper, we consider the following degenerate stochastic differential equations (SDEs in short) with jumps on $\mR^m\times\mR^d$:
\be\label{SDE}
\begin{cases}
X^{(1)}(t,x_1)=x_1+\int_0^tZ^{(1)}(X^{(1)}(s,x_1),X^{(2)}(s,x_2))ds,\\
X^{(2)}(t,x_2)=x_2+\int_0^tZ^{(2)}(X^{(1)}(s,x_1),X^{(2)}(s,x_2))ds+\int_0^t\sigma dL(t),
\end{cases}
\ee
where $X^{(1)}(t,x_1)$ and $X^{(2)}(t,x_2)$ take values in $\mR^m$ and $\mR^d$ respectively, $\sigma$ is an invertible $d\times d$-matrix, $Z^{(1)}\in\cC^2(\mR^{m+d};\mR^m)$, $Z^{(2)}\in\cC^2(\mR^{m+d};\mR^d)$, $(L(t),t\geq0)$ is a purely jump $d$-dimensional L\'{e}vy process with L\'{e}vy measure $\nu(du)=k(u)du$, which has the following representation:
\ce
L(t)=\int_0^t\int_{\cO}u\widetilde{N}(ds,du).
\de
Here $\cO=\{u\in\mR^d;|u|<1\}$, $N(ds,du)$ is the Poisson random measure associated with $L(t)$. It is worth noting that Eq. (\ref{SDE}) can be formulated as
\be\label{SDE convert}
dX(t,x)=Z(X(t,x))dt+(0,\sigma dL(t)),
\ee
where $X(t,x)=(X^{(1)}(t,x_1),X^{(2)}(t,x_2))$, $Z=(Z^{(1)},Z^{(2)})$, and $x=(x_1,x_2)$. For any initial data, Eq. (\ref{SDE}) admits a unique non-explosive solution $(X^{(1)}(t,x_1),X^{(2)}(t,x_2),t\geq0)$ under standard regularity assumptions on the drift coefficients $Z$. Our purpose is to establish an explicit derivative formula for the associated Markov semigroup $P_t$:
\ce
P_t\varphi(x)=\mE\bigl[\varphi(X(t,x))\bigr],\quad \varphi\in\cC_b(\mR^m\times\mR^d),\quad x\in\mR^m\times\mR^d,\quad t\geq0,
\de
where $X(t,x)$ is the solution of Eq. (\ref{SDE convert}) with $X(0,x)=x$, and $\cC_b(\mR^m\times\mR^d)$ denotes the set of bounded continuous functions on $\mR^m\times\mR^d$.

In the present paper, we write the gradient operator on $ \mR^{m+d} $ as $ \nabla = (\nabla^{(1)}, \nabla^{(2)}) $, where $ \nabla^{(1)} $ and $ \nabla^{(2)} $ stand for the gradient operators for the first and the second components respectively, so that $ \nabla f : \mR^{m+d} \to \mR^{m+d} $ for a differentiable function $f$ on $\mR^{m+d}$. Next, for a smooth function $ \xi = (\xi_1, \ldots, \xi_k) : \mR^{m+d} \to \mR^k $, let
\ce
\nabla \xi =
\begin{pmatrix}
\nabla \xi_1 \\
\vdots \\
\nabla \xi_k
\end{pmatrix}, \quad
\nabla^{(i)} \xi =
\begin{pmatrix}
\nabla^{(i)} \xi_1 \\
\vdots \\
\nabla^{(i)} \xi_k
\end{pmatrix}, \quad i = 1, 2.
\de
Then $ \nabla \xi, \nabla^{(1)} \xi, \nabla^{(2)} \xi $ are matrix-valued functions of order $ k \times (m+d), k \times m, k \times d $ respectively. Moreover, for an $ l \times k $-matrix $ M = (M_{ij})_{1 \leq i \leq l, 1 \leq j \leq k} $ and $ v = (v_i)_{1 \leq i \leq k} \in \mR^k $, let $ Mv \in \mR^l $ with $ (Mv)_i = \sum_{j=1}^k M_{ij}v_j, 1 \leq i \leq l $. Finally, we will use $ \|\cdot\| $ to denote the operator norm for linear operators, for instance, $ \|M\| = \sup_{|v|=1} |Mv| $.

In order to establish the derivative formulae and gradient estimates, we will make use of the following assumptions.
\begin{assumption}\label{elliptic}
We assume that $\nabla^{(2)}Z^{(1)}=B_0+B$ for some constant matrix $B_0$ such that
\ce
\<B B_0^Ta,a\>\geq-\varepsilon|B_0^Ta|^2,\quad\forall a\in\mR^m,
\de
holds for some constant $\varepsilon\in[0,1)$. Moreover, there exist two increasing functions $\xi\in\cC([0,T])$ and $\phi\in\cC^1([0,T])$ with $\xi(t)>0$ for $t\in(0,T]$, $\phi(0)=\phi(T)=0$ and $\phi(t)>0$ for $t\in(0,T]$ such that
\ce
\int_0^t\phi(s)K(T,s)B_0B_0^TK^T(T,s)ds\geq\xi(t)I_{m\times m},\quad t\in(0,T].
\de
\end{assumption}

\begin{assumption}\label{coefficient}
	The matrix $\sigma\in\mR^d\otimes\mR^d$ is invertible, and
	\ce
	\|\nabla Z\|\le C,\qquad \|\nabla^2 Z^{(1)}\|\le C,
	\de
	where the second bound means that all second-order partial derivatives of $Z^{(1)}$ are uniformly bounded.
\end{assumption}

\begin{assumption}\label{assumption special}
\begin{enumerate}
\item
The L\'{e}vy measure $\nu$ satisfies $\nu(\cO)=+\infty$.
\item
$k\in\cC_0^1(\cO;\mR^+\setminus\{0\})$, and there exists a positive constant $C$ such that
\ce
|\nabla\log k(u)|\leq C|u|^{-1},\quad u\in\cO.
\de
\item
there exists $\alpha\in(0,2)$ such that
\be\label{growth condition}
\lim_{\varepsilon\downarrow0}\varepsilon^{\alpha-2}\int_{\{|u|\leq\varepsilon\}}|u|^2\nu(du)>0.
\ee
\end{enumerate}
\end{assumption}

Now the derivative formulae and gradient estimates can be established for the semigroup $P_t$ as follows.

\begin{theorem}\label{derivative}
Let $(X^{(1)}(t,x_1),X^{(2)}(t,x_2),t\geq0)$ solve Eq. (\ref{SDE}). Suppose that the assumptions \ref{elliptic}-\ref{assumption special} are satisfied. Then for any $T\in(0,1)$, $h=(h_1,h_2)\in\mR^m\times\mR^d$ and $\varphi\in\cC_b^1(\mR^m\times\mR^d)$, we have
\ce\label{derivative formula}
\nabla_hP_T\varphi(x_1,x_2)=\mE\bigl[\varphi(X^{(1)}(T,x_1),X^{(2)}(T,x_2))M(T,x_1,x_2,h)\bigr],\quad\forall(x_1,x_2)\in\mR^m\times\mR^d,
\de
where $M(T,x_1,x_2,h)$ is a random weight (the Bismut weight) constructed from the divergence operator on the Poisson space; its explicit expression, which relies on the Dirichlet-form structure recalled in Section \ref{sec:proof}, is given there.
Moreover, for any $p>1$, there exists a positive constant $C_p$ such that for any $T\in(0,1)$, $h=(h_1,h_2)\in\mR^m\times\mR^d$,
\ce
\mE\bigl[|M(T,x_1,x_2,h)|^p\bigr]&\leq&C_p \bigl(|h_1|^p+|h_2|^p\bigr)\Bigl[
\Lambda^p(T)T^{-p/\alpha}+ \Lambda^p(T)\no\\
&&+ \frac{T^{4p}}{\xi^{p}(T)} + \frac{T^{5p}}{\xi^{2p}(T)}\no\\
&&+ \frac{T^{6p}}{\bigl(\int_0^T\xi^2(s)ds\bigr)^p}+ \frac{\bigl(\int_0^T\xi(s)ds\bigr)^p T^{4p}}{\bigl(\int_0^T\xi^2(s)ds\bigr)^p}\no\\
&&+ T^{2p} + T^{3p}
\Bigr],
\de
where
\ce
\Lambda(T)=1 + \frac{1}{\xi(T)} + \frac{\int_0^T \xi(s)ds}{\int_0^T \xi^2(s)ds}.
\de
Consequently, for any $q>1$ with $1/p+1/q=1$, applying H\"older's inequality yields the gradient estimate
\ce
|\nabla_hP_T\varphi(x_1,x_2)|
&\leq& C_p\bigl(|h_1|^p+|h_2|^p\bigr)^{1/p}\bigl(P_T|\varphi(x_1,x_2)|^q\bigr)^{1/q}\Bigl[
\Lambda^p(T)T^{-p/\alpha}+ \Lambda^p(T)\\
&&+ \frac{T^{4p}}{\xi^{p}(T)} + \frac{T^{5p}}{\xi^{2p}(T)}\\
&&+ \frac{T^{6p}}{\bigl(\int_0^T\xi^2(s)ds\bigr)^p}+ \frac{\bigl(\int_0^T\xi(s)ds\bigr)^p T^{4p}}{\bigl(\int_0^T\xi^2(s)ds\bigr)^p}\\
&&+ T^{2p} + T^{3p}
\Bigr]^{1/p}.
\de
\end{theorem}

We now turn to the background and motivation that lead to the above result. When $m = d$, $\sigma = I$ and
\ce
Z^{(1)}(x_1, x_2) = \nabla H(x_1, \cdot)(x_2), \quad Z^{(2)}(x_1, x_2) = -\nabla H(\cdot, x_2)(x_1) - F(x_1, x_2)\nabla H(x_1, \cdot)(x_2)
\de
for some functions $H$ and $F$, Eq. (\ref{SDE}) goes back to the stochastic Hamiltonian systems with jumps
\be\label{SHS}
\begin{cases}
dX^{(1)}(t,x_1) = \nabla H(X^{(1)}(t,x_1), \cdot)(X^{(2)}(t,x_2))dt, \\
dX^{(2)}(t,x_2) = -\{ \nabla H(\cdot, X^{(2)}(t,x_2))(X^{(1)}(t,x_1)) \\
\quad\quad\quad\quad\quad\quad+ F(X^{(1)}(t,x_1), X^{(2)}(t,x_2))\nabla H(X^{(1)}(t,x_1), \cdot)(X^{(2)}(t,x_2)) \} dt + dL(t),
\end{cases}
\ee
with Hamiltonian function $H$. For such jump-type systems, various long-time behaviors have been studied: exponential ergodicity for L\'{e}vy-driven Langevin dynamics with singular potentials was established in \cite{BFW}; Wasserstein-type exponential contraction rates were derived in \cite{BW} using coupling methods; exponential contraction rates for degenerate infinite-dimensional SDEs with L\'{e}vy noises were obtained in \cite{LWZ}; and $L^2$-exponential ergodicity for $\alpha$-stable noise driven systems was proved in \cite{BW2}. In particular, if $H(x_1, x_2) = V(x_1) + \frac{1}{2}|x_2|^2$ and $F \equiv c$ for some constant $c$, Eq. (\ref{SHS}) is known as the stochastic damping Hamiltonian system with jumps.

In recent years, numerous studies have addressed derivative formulae for jump-type SDEs: for example, those driven by general Poisson jump processes via stochastic diffeomorphism flows and Girsanov's transformation (\cite{T}); those driven by subordinated Brownian motion using conditional Malliavin calculus (\cite{Z1,WXZ2}); and those driven by $\alpha$-stable-like noise via Bismut's approach (\cite{WXZ1}).

More recently, Bouleau and Denis have systematically proposed a novel technique---the lent particle method---within Malliavin calculus for jumps, which enables one to relax certain restrictions on the intensity that are inherent in most earlier approaches. This method is particularly convenient because its gradient operator is local (see \cite{BD} and references therein). In our previous work \cite{RZ}, we already applied this method to derive Bismut--Elworthy--Li type derivative formulae for nondegenerate jump SDEs.

Nevertheless, all these results were obtained exclusively under the elliptic (nondegenerate) framework, and the degenerate case remains largely open. The present paper aims to contribute to filling this gap. It is worth pointing out that even for Brownian-driven classical SDEs, the explicit Bismut--Elworthy--Li type formulae in degenerate settings are still rare. So far, only a few specific cases have been treated rigorously: Bessel processes (\cite{A}), Gruschin-type semigroups (\cite{W1}), and Kohn-Laplacian semigroups (\cite{W2}). Bismut--Elworthy--Li type formulae have also been established in \cite{WZ,RW} for distribution-dependent and non-distribution-dependent SDEs with other degenerate coefficients. By contrast, rigorous results for jump-driven degenerate SDEs are far more scarce, and systematic treatments of such models remain largely absent from the literature.

Moreover, the lent particle method has recently been further exploited to handle some specific nonlocal degenerate operators: derivative formulae for nonlocal Gruschin's type semigroups were obtained in \cite{Zhang2026a}, and for nonlocal Kohn--Laplacian-type semigroups in \cite{Zhang2026b}. These works represent the only applications of the lent particle method to degenerate jump-noise settings so far. In spite of these initial efforts, a systematic treatment for general nonlocal degenerate operators---in particular, for stochastic Hamiltonian systems with jumps---is still missing. The present paper continues this line of investigation and establishes Bismut--Elworthy--Li type derivative formulae as well as the corresponding gradient estimates for such systems. Together with \cite{Zhang2026a} and \cite{Zhang2026b}, this work forms a series of our ongoing efforts to apply the lent particle method to derive Bismut--Elworthy--Li type formulae under degenerate noises.

The remainder of this paper is organized as follows. In Section \ref{sec:app}, we apply Theorem \ref{derivative} to a stochastic damped Hamiltonian system driven by $\alpha$-stable-like jump noise and derive explicit gradient bounds. In Section \ref{sec:proof}, we prove Theorem \ref{derivative} by first recalling the lent-particle calculus on the Poisson space, then constructing the Bismut weight $M$ explicitly, and finally establishing the derivative formula together with its $L^p$-estimate and the resulting gradient estimate.

\section{An Application to Stochastic Damped Hamiltonian Systems with Jumps}\label{sec:app}

In this section, we apply the derivative formula and gradient estimates established in Theorem \ref{derivative} to a class of stochastic damped Hamiltonian systems driven by $\alpha$-stable-like jump noise. We verify that the model satisfies the assumptions \ref{elliptic}-\ref{assumption special} under natural regularity conditions, and derive explicit gradient bounds for the associated Markov semigroup. We also briefly discuss extensions to more general degenerate drift structures and to distribution-dependent systems, in the spirit of the related works on degenerate diffusions \cite{WZ} and mean-field SDEs \cite{RW}.

\subsection{A Stochastic Damped Hamiltonian System with Jump Noise}

Consider the following stochastic Hamiltonian system on $\mR^d \times \mR^d$ (i.e. $m=d$ in the general framework of the present paper):
\be\label{5.1}
\begin{cases}
dX^{(1)}(t,x_1) = \nabla_y H(X^{(1)}(t,x_1), X^{(2)}(t,x_2)) dt, \\
dX^{(2)}(t,x_2) = -\nabla_x H(X^{(1)}(t,x_1), X^{(2)}(t,x_2)) dt - \gamma \nabla_y H(X^{(1)}(t,x_1), X^{(2)}(t,x_2)) dt + \sigma dL(t),
\end{cases}
\ee
where
\begin{enumerate}[(i)]
    \item $X^{(1)}(t,x_1) \in \mR^d$ denotes the position process and $X^{(2)}(t,x_2) \in \mR^d$ the momentum process;
    \item $H: \mR^d \times \mR^d \to \mR$ is the Hamiltonian function of the form
    \ce
    H(x_1,x_2) = V(x_1) + \frac{1}{2} \langle M x_2, x_2 \rangle,
    \de
    with $V \in \cC^2(\mR^d)$ the potential energy and $M \in \mR^{d \times d}$ a constant symmetric positive definite mass matrix;
    \item $\gamma > 0$ is the constant damping coefficient;
    \item $\sigma \in \mR^{d \times d}$ is an invertible noise intensity matrix;
    \item $L(t)$ is a $d$-dimensional $\alpha$-stable-like L\'{e}vy process with L\'{e}vy measure supported on $\cO$:
    \ce
    \nu(du) = k(u) du = \frac{a(u)}{|u|^{d+\alpha}} du, \quad \alpha \in (0,2),
    \de
    where $a(u)=a(-u)$ is a smooth positive function satisfying
    \ce
    0 < a_0 \leq a(u) \leq a_1 < \infty, \quad |\nabla a(u)| \leq a_2, \quad \forall u \in \cO,
    \de
    for some constants $a_0,a_1,a_2>0$ (cf. \cite{CK}).
\end{enumerate}

In the notation of the introduction, the drift coefficients are given by
\ce
Z^{(1)}(x_1,x_2) &=& \nabla_{x_2} H(x_1,x_2) = M x_2, \\ 
Z^{(2)}(x_1,x_2) &=& -\nabla_{x_1} H(x_1,x_2) - \gamma \nabla_{x_2} H(x_1,x_2) = -\nabla V(x_1) - \gamma M x_2.
\de
This model is the jump-noise counterpart of the degenerate diffusion Hamiltonian system studied in \cite{WZ}, and describes a mechanical system subject to random impulsive forces of $\alpha$-stable type.

\subsection{Verification of Structural Assumptions}

We impose the following standard condition on the potential:

\ba\label{V}
The potential $V \in \cC^2(\mR^d)$ has bounded Hessian, i.e. there exists a constant $C_V > 0$ such that
\ce
\|\nabla^2 V(x)\| \leq C_V, \quad \forall x \in \mR^d.
\de
\ea

Under the assumption \ref{V}, we verify that system (\ref{5.1}) satisfies all assumptions in Theorem \ref{derivative}.

\subsubsection*{Verification of the assumption \ref{elliptic}}
By direct computation, $\nabla^{(2)} Z^{(1)}(x_1,x_2) = M$. Since $M$ is constant and positive definite, we take $B_0 = M$ and $B \equiv 0$. Then the domination condition
\ce
\langle B B_0^T a, a \rangle \ge -\varepsilon |B_0^T a|^2, \quad \forall a \in \mR^d,
\de
holds trivially with $\varepsilon = 0$.

Next, we check the controllability condition. The first variation process $K(t,s)$ satisfies
\ce
\frac{d}{dt} K(t,s) = (\nabla^{(1)} Z^{(1)})(X^{(1)}(t,x_1), X^{(2)}(t,x_2)) K(t,s), \quad K(s,s) = I_{d\times d}.
\de
Since $Z^{(1)}(x_1,x_2)=M x_2$ is independent of $x_1$, we have $\nabla^{(1)} Z^{(1)} \equiv 0$, hence $K(t,s) \equiv I_{d\times d}$ for all $t\ge s$. Taking the weight function
\ce
\phi(t) = \frac{t(T-t)}{T^2}, \quad t\in[0,T],
\de
which satisfies $\phi(0)=\phi(T)=0$ and $\phi(t)>0$ for $t\in(0,T)$, we obtain
\ce
\int_0^t \phi(s) K(T,s) B_0 B_0^T K^T(T,s) ds = M^2 \int_0^t \frac{s(T-s)}{T^2} ds.
\de
Let $\lambda_{\min}(M^2)>0$ be the smallest eigenvalue of $M^2$. Define
\ce
\xi(t) = \lambda_{\min}(M^2) \int_0^t \frac{s(T-s)}{T^2} ds = \lambda_{\min}(M^2) \cdot \frac{3T t^2 - 2t^3}{6T^2}.
\de
Then $\xi(t)>0$ for $t\in(0,T]$, and
\ce
\int_0^t \phi(s) K(T,s) B_0 B_0^T K^T(T,s) ds \ge \xi(t) I_{d\times d}, \quad t\in(0,T].
\de
Thus the assumption \ref{elliptic} is satisfied.

\subsubsection*{Verification of the assumption \ref{coefficient}}
The invertibility of $\sigma$ is given. For the boundedness of the first derivative $\nabla Z$, we have
\ce
\nabla Z^{(1)} = \begin{pmatrix} 0 & M \end{pmatrix},
\de
which is constant and hence uniformly bounded, and
\ce
\nabla Z^{(2)} = \begin{pmatrix} -\nabla^2 V(x) & -\gamma M \end{pmatrix},
\de
which is uniformly bounded by the assumption \ref{V} and the constancy of $\gamma M$. Thus $\|\nabla Z\|\le C$ for some constant $C>0$. It remains to check the second derivative of $Z^{(1)}$. Since $Z^{(1)}(x_1,x_2)=Mx_2$ is linear in both variables, all its second-order partial derivatives vanish identically, namely $\nabla^2 Z^{(1)}\equiv0$, and the bound $\|\nabla^2 Z^{(1)}\|\le C$ holds trivially. Therefore the assumption \ref{coefficient} is satisfied.

\subsubsection*{Verification of the assumption \ref{assumption special}}
The L\'{e}vy measure $\nu(du)=a(u)|u|^{-d-\alpha}du$ satisfies:
\begin{enumerate}[(i)]
    \item $\nu(\cO)=+\infty$ since $\alpha\in(0,2)$;
    \item $k(u)=a(u)|u|^{-d-\alpha}\in \cC_0^1(\mR^d\setminus\{0\})$, and
    \ce
    |\nabla \log k(u)| = | \frac{\nabla a(u)}{a(u)} - (d+\alpha)\frac{u}{|u|^2} | \le \frac{a_2}{a_0} + \frac{d+\alpha}{|u|} \le C |u|^{-1}
    \de
    for some $C>0$;
    \item Polar coordinates give
    \ce
    \int_{\{|u|\le \varepsilon\}} |u|^2 \nu(du) = \int_{\{|u|\le \varepsilon\}} a(u) |u|^{2-d-\alpha} du \asymp \varepsilon^{2-\alpha}, \quad \varepsilon\to0,
    \de
    so that
    \ce
    \lim_{\varepsilon\downarrow0} \varepsilon^{\alpha-2} \int_{\{|u|\le \varepsilon\}} |u|^2 \nu(du) > 0.
    \de
\end{enumerate}
Thus the assumption \ref{assumption special} holds.

\subsection{General Controllability Conditions for Nonlinear Drifts}

The above verification used the fact that $Z^{(1)}$ is linear in $x_2$ and independent of $x_1$. In more general situations where $\nabla^{(2)}Z^{(1)}$ is not constant or $\nabla^{(1)}Z^{(1)}$ is non-zero, the assumption \ref{elliptic} can still be verified under standard controllability conditions, as demonstrated in \cite[Section 4]{WZ}. We briefly recall two typical cases that are covered by our framework.

\begin{enumerate}
\item[\rm (I)] \textbf{Rank condition on the principal part.} Suppose $\operatorname{Rank}(B_0)=m$ and $\nabla^{(1)}Z^{(1)}$ is bounded (which is implied by the assumption \ref{coefficient}). Taking again
\ce
\phi(t)=\frac{t(T-t)}{T^2}, \quad t\in[0,T],
\de
there exist constants $c_1,c_2>0$ such that
\ce
\int_0^t \phi(s)K(T,s)B_0B_0^TK^T(T,s)ds \ge \xi(t) I_{m\times m},
\de
with
\ce
\xi(t)= c_1 \int_0^t \phi(s) e^{-c_2(T-s)}ds, \quad t\in[0,T].
\de
Indeed, this follows from the boundedness of $K$ and the lower bound $|B_0^Ta|\ge c|a|$.

\item[\rm (II)] \textbf{Kalman condition for constant $A:=\nabla^{(1)}Z^{(1)}$.} Assume $A$ is constant and there exists an integer $k\ge0$ such that
\ce
\operatorname{Rank}[B_0, AB_0, \dots, A^k B_0]=m.
\de
Then, with the same $\phi(t)=t(T-t)/T^2$, there exist constants $c_1,c_2>0$ such that
\ce
\int_0^t \phi(s)K(T,s)B_0B_0^TK^T(T,s)ds \ge \xi(t) I_{m\times m},
\de
where
\ce
\xi(t)= \frac{c_1 (t\wedge 1)^{2(k+1)}}{Te^{c_2 T}},\qquad t\in[0,T].
\de
This is a standard consequence of the Kalman controllability condition; see \cite[Theorem 4.2]{WZ}.
\end{enumerate}

Our concrete system (\ref{5.1}) with $Z^{(1)}(x_1,x_2)=M x_2$ corresponds to Case (I) (or Case (II) with $k=0$ and $A=0$), since $\operatorname{Rank}(M)=d=m$. Therefore, the above discussion shows that the assumptions of Theorem \ref{derivative} are not only satisfied for the linear model but also for a wide class of nonlinear degenerate systems, provided the drift coefficients satisfy the corresponding controllability and domination conditions.

\subsection{Explicit Derivative Formulae and Gradient Estimates}

Since all assumptions are verified for system (\ref{5.1}), we can apply Theorem \ref{derivative} directly to obtain the following Bismut--Elworthy--Li type derivative formulae and gradient estimates.

\begin{proposition}
Let $P_T$ be the Markov semigroup associated with system (\ref{5.1}). Under the assumption \ref{V}, for any $T\in(0,1)$, $h=(h_1,h_2)\in\mR^d\times\mR^d$ and $\varphi\in \cC_b^1(\mR^d\times\mR^d)$, the derivative formula
\ce
\nabla_h P_T \varphi(x_1,x_2) = \mE\bigl[ \varphi(X^{(1)}(T,x_1), X^{(2)}(T,x_2))M(T,x_1,x_2,h) \bigr]
\de
holds, where $M$ is the Bismut weight from Theorem \ref{derivative}.
Moreover, for any $p>1$, there exists a constant $C_{p,d,\alpha,\gamma,M,\sigma}>0$ such that
\ce
\mE\bigl[|M(T,x_1,x_2,h)|^p\bigr]
\le C_{p,d,\alpha,\gamma,M,\sigma}(|h_1|^p+|h_2|^p)T^{-p(1+1/\alpha)}.
\de
Consequently, for any $q>1$ with $1/p+1/q=1$,
\ce
|\nabla_h P_T\varphi(x_1,x_2)|
\le C'_{p,d,\alpha,\gamma,M,\sigma} (|h_1|^p+|h_2|^p)^{1/p}(P_T|\varphi(x_1,x_2)|^q)^{1/q}T^{-1-1/\alpha}.
\de
\end{proposition}

\begin{proof}
We apply the estimate in Theorem \ref{derivative} to system (\ref{5.1}). For this system, $\xi(t)=\eta\lambda(t)$ with $\eta:=\lambda_{\min}(M^2)>0$ a constant and $\lambda(t):=(3T t^2 - 2t^3)/(6T^2)$, and we have
\ce
\xi(T)=\eta\frac{T}{6},\qquad
\int_0^T\xi(t)dt=\eta\frac{T^2}{12},\qquad
\int_0^T\xi^2(t)dt=\eta^2\frac{13}{1260}T^3.
\de
Consequently,
\ce
\Lambda(T)=1+\frac{1}{\xi(T)}+\frac{\bigl(\int_0^T\xi(t)\,dt\bigr)}{\bigl(\int_0^T\xi^2(t)\,dt\bigr)}
=1+\frac{6}{\eta T}+\frac{105}{13\eta T}
\le \frac{C}{T},
\de
for some constant $C>0$ depending only on $\eta$.

Substituting these into the estimate from Theorem \ref{derivative} gives
\ce
&&\mE\bigl[|M(T,x_1,x_2,h)|^p\bigr]\\
&\le& C_{p,d,\alpha,\gamma,M,\sigma} \bigl(|h_1|^p+|h_2|^p\bigr)
\Bigl[
\Lambda^p(T)T^{-p/\alpha}+\Lambda^p(T)\\
&& +\frac{T^{4p}}{\xi^p(T)}+\frac{T^{5p}}{\xi^{2p}(T)}
+\frac{T^{6p}}{\bigl(\int_0^T\xi^2(t)dt\bigr)^p}
+\frac{\bigl(\int_0^T\xi(t)dt\bigr)^pT^{4p}}{\bigl(\int_0^T\xi^2(t)dt\bigr)^p}
+T^{2p}+T^{3p}
\Bigr] \\
&\le& C_{p,d,\alpha,\gamma,M,\sigma} \bigl(|h_1|^p+|h_2|^p\bigr)
\Bigl[
T^{-p(1+1/\alpha)} + T^{-p} + T^{2p} + T^{3p}
\Bigr] \\
&\le& C'_{p,d,\alpha,\gamma,M,\sigma} (|h_1|^p+|h_2|^p)
T^{-p(1+1/\alpha)}.
\de
It remains to deduce the gradient estimate. Its general form, with the bracket involving $\Lambda(T)$ and $\xi(T)$, has already been established in Theorem \ref{derivative} by applying H\"older's inequality with $1/p+1/q=1$; substituting the three identities above into that general estimate (equivalently, taking the $p$-th root of the preceding moment bound) gives
\ce
&&|\nabla_h P_T\varphi(x_1,x_2)|\\
&=& |\mE\bigl[\varphi(X^{(1)}(T,x_1),X^{(2)}(T,x_2))M(T,x_1,x_2,h)\bigr]| \\
&\le& (P_T|\varphi(x_1,x_2)|^q)^{1/q}
(\mE\bigl[|M(T,x_1,x_2,h)|^p\bigr])^{1/p} \\
&\le& C''_{p,d,\alpha,\gamma,M,\sigma}
(|h_1|^p+|h_2|^p)^{1/p}
(P_T|\varphi|^q)^{1/q}
T^{-(1+1/\alpha)}.
\de
Then we complete the proof.
\end{proof}

\br
This estimate reveals the short-time regularity of the semigroup. For small $T$, the gradient bound blows up at the rate $T^{-1/\alpha-1}$. The dominant contribution comes from the $T^{-1/\alpha}$ factor inherent in the $\alpha$-stable noise, while the additional $T^{-1}$ factor originates from the controllability weight $\Lambda(T)\sim T^{-1}$. This rate is consistent with the scaling property of $\alpha$-stable processes.
\er

\begin{remark}
\label{rem:state_dependent_mass}
Following the framework of \cite{WZ}, the above derivative formulae and gradient estimates can be extended to systems with a state-dependent mass matrix $M(x)$. Writing $M(x)=M_0+M_1(x)$ with constant positive definite $M_0$, the controllability condition in the assumption \ref{elliptic} remains valid provided the perturbation $M_1(x)$ is suitably dominated by $M_0$ (in the sense of the ellipticity-type condition therein). In this case, the first variation process $K(t,s)$ is no longer the identity, but its exponential growth is controlled by the boundedness of $\nabla M$. Moreover, combining the Bismut--Elworthy--Li type formula with the Young inequality, one can derive log-Harnack and shift-Harnack inequalities for the semigroup $P_T$ along the lines of \cite{DS} for stochastic Hamiltonian systems with jumps. These inequalities further imply the existence of a smooth transition density and entropy-cost estimates for the invariant measure when the system is ergodic.
\end{remark}

\begin{remark}
\label{rem:mean_field}
The lent particle method used in this paper is also adaptable to distribution-dependent (mean-field) stochastic Hamiltonian systems with jumps, in the spirit of \cite{RW}. The main additional difficulty lies in combining the Malliavin calculus for Poisson random measures with the Lions derivative calculus on the Wasserstein space. This is left for future research.
\end{remark}

\section{Proof of the Main Result}\label{sec:proof}

This section is devoted to the proof of Theorem \ref{derivative}. We first recall the Malliavin calculus with jumps built upon the lent particle method, a framework introduced by Bouleau and Denis that forms the basis of our construction. We then construct the Bismut weight $M$ explicitly, and finally establish the derivative formula together with its $L^p$-estimate and the resulting gradient estimate.

\subsection{Dirichlet structure on the bottom space}

We start from a bottom space $(\Xi,\cG,\nu)$, where $\Xi$ is a separable Hausdorff space, $\cG$ its Borel $\sigma$-algebra and $\nu$ a $\sigma$-finite and diffuse measure on $(\Xi,\cG)$. Let $(\mathbf{d},e)$ be a local symmetric Dirichlet form on $L^2(\nu)$ which admits a carr\'{e} du champ operator $\gamma$. That is to say, $\gamma$ is the unique positive, symmetric and continuous bilinear form from $\mathbf{d}\times\mathbf{d}$ to $L^1(\nu)$ such that
\ce
e(f,g)=\frac12\int\gamma[f,g]d\nu,\quad\forall f,g\in\mathbf{d}.
\de
Moreover, we suppose that there exist $\{k_n,n\in\mN\}\subset\mathbf{d}$ and $A_n\uparrow\Xi$ such that $k_n1_{A_n}\uparrow 1$ and $\gamma[k_n]1_{A_n}=0$. The structure $(\Xi,\cG,\nu,\mathbf{d},\gamma)$ is called the bottom structure.

Since $\mathbf{d}$ is separable, the bottom Dirichlet structure admits a gradient operator, i.e., there exists a separable Hilbert space $H$ and a linear map $D$ from $\mathbf{d}$ into $L^2(\nu;H)$ such that
\ce
\gamma[u]=\|Du\|_H^2,\quad\forall u\in\mathbf{d}.
\de
Let $(R,\cR,\rho)$ be another probability space such that the vector space $L^2(R,\cR,\rho)$ is infinite dimensional. Take $H=L_0^2(R,\cR,\rho)=\{g\in L^2(R,\cR,\rho);\int_Rg(r)\rho(dr)=0\}$. The corresponding gradient will be denoted by $\flat$, and we assume without any loss of generality that constants belong to $\mathbf{d}_{loc}$ (see \cite[Chapter I, Definition 7.1.3]{BF}), and hence $1^{\flat}=0$.

\subsection{Space-time setting and the Dirichlet structure on the upper space}

From now on we set $X=\mR^+\times \Xi$, $\cX=\cB(\mR^+)\times \cG$ and $\mu=dt\times\nu$. Define the Dirichlet structure on $(X,\cX,\mu)$ to be the product of the trivial one on $(L^2(\mR^+, dt),0)$ and $(\mathbf{d}, e)$ and we keep the same notations $\mathbf{d}$, $e$, $\gamma$ and $\flat$, etc, for operators corresponding to this new Dirichlet form but note that they act only on the second variable.

It is known that $X$ is totally ordered (see \cite[Section 1, Theorem 11]{DM}) and we denote by $\prec$ such a total order relation. Set
\ce
\Omega:=\{\omega=\sum_{i=1}^\infty\e_{y_i};\quad y_i\in X,\quad\forall i,\quad\text{and}\quad y_1\prec y_2\prec\cdots\prec y_n\prec\cdots\}.
\de
Let $N$ be the Poisson random measure with intensity $\mu$ defined on $(\Omega,\cA, \mP)$ where $N(\omega)=\omega$, $\cA$ is the $\sigma$-algebra generated by $N$ and $\mP$ the law of $N$.

We now introduce the creation and annihilation operator $\varepsilon^+$ and $\varepsilon^-$:
\ce
\forall(t,u)\in\mR^+\times\Xi,\quad\forall\omega\in\Omega,\\
\varepsilon_{(t,u)}^+(\omega)=\omega1_{\{(t,u)\in \operatorname{supp}\omega\}}+(\omega+\varepsilon_{(t,u)})1_{\{(t,u)\notin \operatorname{supp}\omega\}},\\
\forall(t,u)\in\mR^+\times\Xi,\quad\forall\omega\in\Omega,\\
\varepsilon_{(t,u)}^-(\omega)=\omega1_{\{(t,u)\notin \operatorname{supp}\omega\}}+(\omega-\varepsilon_{(t,u)})1_{\{(t,u)\in \operatorname{supp}\omega\}}.
\de
Denote $\mP_N:=\mP(d\omega)N_{\omega}(dt,du)$.  Then it is well known (see \cite[Lemma 4.2]{BD}) that the map $(\omega,(t,u))\mapsto (\varepsilon_{(t,u)}^+\omega,(t,u))$
sends $\mP_N$-negligible sets to $\mP\times\mu$-negligible ones, and the map $(\omega,(t,u))\mapsto (\varepsilon_{(t,u)}^-\omega,(t,u))$
sends $\mP\times\mu$-negligible sets to $\mP_N$-negligible ones.

If $N(\omega)=\sum_{i=1}^{\infty}\varepsilon_{y_i}$, then define
\ce
N\odot\rho(\omega,\hat{\omega}):=\sum_{i=1}^{\infty}\varepsilon_{(y_i,r_i(\hat{\omega}))},
\de
where $(r_i)$ is a sequence of i.i.d random variables independent of $N$ whose common law is $\rho$ and which are defined on some probability space $(\widehat{\Omega},\widehat{\cA},\widehat{\mP})$. Hence $N\odot\rho$ is defined on the product probability space $(\Omega,\cA,\mP)\times(\widehat{\Omega},\widehat{\cA},\widehat{\mP})$. It is a Poisson random measure on $X\times R$ with compensator $\mu\times\rho$ which is called the marked Poisson random measure.

The marked Poisson random measure $N\odot\rho$ has the following useful features in connection with $N$ and the measure $\widehat{\mP}$ which is crucial in the following section. The proof is similar to that of \cite[Corollary 4.16]{BD}, and we omit it.

\begin{proposition}\label{marked Poisson measure}
Let $F$, $G:\Omega\times X\times R\rightarrow\mR$ be $\cA\otimes\cX\otimes\cR$-measurable functions such that
\ce
&&\int_0^{\infty}\int_{\Xi}\int_R(|F(t,u,r)|+|F(t,u,r)|^2)\rho(dr)N(dt,du)<\infty,\quad\mP-\text{a.e.},\\
&&\int_RF(t,u,r)\rho(dr)=0,\quad\mP_N-\text{a.e.}
\de
and
\ce
&&\int_0^{\infty}\int_{\Xi}\int_R(|G(t,u,r)|+|G(t,u,r)|^2)\rho(dr)N(dt,du)<\infty,\quad\mP-\text{a.e.},\\
&&\int_RG(t,u,r)\rho(dr)=0,\quad\mP_N-\text{a.e.}
\de
Then the following relation holds $\mP_N$-a.e.
\ce
&&\widehat{\mE}\bigl[(\int_0^{\infty}\int_{\Xi}\int_RF(t,u,r)N\odot\rho(dt,du,dr))(\int_0^{\infty}\int_{\Xi}\int_RG(t,u,r)N\odot\rho(dt,du,dr))\bigr]\\
&=&\int_0^{\infty}\int_{\Xi}\int_RF(t,u,r)G(t,u,r)\rho(dr)N(dt,du).
\de
Here and in the sequel, $\int_{\alpha}^{\beta}=\int_{(\alpha,\beta]}$ for $\alpha<\beta$, and $\widehat{\mE}$
denotes the expectation with respect to $\widehat{\mP}$.
\end{proposition}

On the product structure $(X_n,\cX_n,\mu_n,\mathbf{d}_n,\gamma_n)
=(X,\cX,\mu,\mathbf{d},\gamma)^n$ the Dirichlet form is defined by
\ce
e_n(f)=\frac12\int_{X_n}\gamma_n[f]d\mu_n.
\de
We refer to \cite[Chapter V]{BF} for a detailed account of the theory of this product Dirichlet structure. Let $\{p_t;t\in\mR^+\}$ be the semigroup associated to $e$ and $\{p^n_t;t\in\mR^+\}$ that associated to $e_n$. For $F\in L^2(\mP)$ with the chaos decomposition
\ce
F=\mE\bigl[F\bigr]+\sum_{n=1}^{\infty}I_n(f_n),
\de
define
\ce
P_tF=\mE\bigl[F\bigr]+\sum_{n=1}^\infty I_n(p_t^nf_n).
\de
Then $\{P_t;t\in\mR^+\}$ on $L^2(\mP)$ forms a symmetric strongly continuous
semigroup whose infinitesimal generator is given by
\ce
\cD(A)=\{F\in L^2(\mP):\lim_{t\downarrow 0}\frac{P_tF-F}{t}\quad\text{exists in}\quad L^2(\mP)\}
\de
and
\ce
A[F]=\lim_{t\downarrow 0}\frac{P_tF-F}{t},\quad\forall F\in\cD(A).
\de
Define
\ce
\mD=\{F\in L^2(\mP):\lim_{t\downarrow 0}\<\frac{F-P_tF}{t},F\>_{L^2(\mP)}<\infty\}
\de
and
\ce
\cE(F)=\lim_{t\downarrow0}\<\frac{F-P_tF}{t},F\>_{L^2(\mP)},\quad\forall F\in\mD.
\de
Then $(\mD,\cE)$ is a local symmetric Dirichlet form on $L^2(\mP)$ with a carr\'e du champ operator $\Gamma$ given by
\ce
\Gamma[F,G]=\widehat{\mE}\bigl[F^{\sharp}G^{\sharp}\bigr],
\de
where
\ce\label{gradient}
F^\sharp=\int_0^{\infty}\int_{\Xi}\int_R\varepsilon^-(\varepsilon^+F)^{\flat}dN\odot\rho.
\de
Moreover, we have
\ce\label{carre du champ}
\Gamma[F]=\widehat{\mE}\bigl[(F^{\sharp})^2\bigr]=\int_0^{\infty}\int_{\Xi}\varepsilon^-(\gamma[\varepsilon^+F])dN,\quad\forall F\in\mD.
\de

We now recall the definition of the divergence operator $\delta_{\sharp}$.

\begin{definition}
Let $\mH=L^2(\widehat{\mP})$, the operator $\delta_{\sharp}:L^2(\mP;\mH)\rightarrow\mD$ is defined as the adjoint operator of the gradient $F\in\mD\mapsto F^{\sharp}\in L^2(\mP;\mH)$, and we denote by $dom\delta_{\sharp}$ the domain of $\delta_{\sharp}$.
\end{definition}

We will need the following result which is due to \cite[Proposition 5.6]{BD}.

\begin{proposition}\label{divergence}
Let $F=F(\omega,\hat{\omega})$ be $\cA^{\odot}$-measurable, where $\cA^{\odot}$ is the $\sigma$-field on $\Omega\times\widehat{\Omega}$ generated by $N\odot\rho$, $F\in dom\delta_{\sharp}$, then
\ce
\delta_{\sharp}F=\int_0^{\infty}\int_{\Xi}\varepsilon_{(\alpha,u)}^-(\widehat{\mE}\bigl[\delta_{\flat}(\varepsilon_{(\alpha,u,r)}^+F)(\alpha,u)\bigr])N(d\alpha,du),
\de
where $\varepsilon^+$ is relative to the marked Poisson random measure $N\odot\rho$ under $\mP\times \widehat{\mP}$ hence adds a point $(\alpha, y,r)$ while operator $\varepsilon$ is relative to $N$.
\end{proposition}

\subsection{Construction of the Bismut weight}

From now on, we assume that
\ce
(\Xi,\cG,\nu)=(\cO,\cB(\cO),k(u)du).
\de
Fix $\varepsilon\in(0,1)$ and let $\zeta(u)$ be a smooth real function with
\ce
\zeta(u):=\zeta_{\varepsilon}(u)=
\begin{cases}
|u|^3,&|u|\leq\frac{\varepsilon}{3},\\
0,&|u|>\frac{2\varepsilon}{3},
\end{cases}
\de
and
\ce
|\nabla\zeta_{\varepsilon}(u)|\leq C|u|^2,
\de
where $C$ is a positive constant independent of $\varepsilon$. The specific value of $\varepsilon$ will be specified later in accordance with the needs of the proof of Theorem \ref{derivative}. Denote by $\cC_0^\infty(\cO)$ the set of $\cC^\infty$-functions defined on $\cO$ and with compact support and by $H$ the subspace consisting of $f\in L^2(\nu)\cap L^1(\nu)$ such that $f\in\cC_0^{\infty}(\cO)$. Then the bilinear form
\ce
e(\varphi,\psi)=\frac{1}{2}\sum_{i=1}^d\int_{\cO}\zeta(u)\partial_i\varphi(u)\partial_i\psi(u)k(u)du,\quad\forall \varphi,\psi\in H
\de
is closable and its closure, which will be denoted by $(\mathbf{d},e)$, is a local symmetric Dirichlet form on $L^2(\nu)$ which admits a carr\'{e} du champ operator $\gamma$ given by
\ce
\quad\gamma[\varphi,\psi]=\sum_{i=1}^d\zeta\partial_i\varphi\partial_i\psi,\quad\forall \varphi,\psi\in\mathbf{d}.
\de

Since
\ce
\gamma[\psi](u)=\zeta(u)\sum_{i=1}^d(\partial_i\psi(u))^2,\quad\forall \psi\in\cC_0^{\infty}(\cO),
\de
and the identity map $j$ belongs to $\mathbf{d}$, we have $\gamma[j,j^*](u)=\zeta(u)I$. Therefore,
\ce
j^{\flat}(u,r)=\zeta^{1/2}(u)\xi(r).
\de
Here
\ce
\xi(r)=
\left(\begin{array}{c}
\xi^1(r) \\
\xi^2(r) \\
\vdots \\
\xi^d(r)
\end{array}\right),
\de
where $\xi^i\in L^2(\rho)$, $i=1,2,\cdots,d$, satisfy
\ce
\int_{R}\xi^i(r)\rho(dr)=0
\de
and
\ce
\int_{R}\xi^i(r)\xi^j(r)\rho(dr)=\delta_{ij}
\de
for all $i,j=1,\cdots,d$. Consequently,
\ce
\psi^{\flat}(u,r)=\zeta^{1/2}(u)\sum_{i=1}^d\partial_i\psi(u)\xi^i(r),
\de
and for a $U(u,r)\in L^2(\nu\times\rho)$ which is $C^1$ in the first variable,
\be\label{theta}
[\delta_\flat U](u)=-\int_{R}[\theta(u,r)U(u,r)+\zeta^{1/2}(u)\sum_{i=1}^d\partial_iU(u,r)\xi^i(r)]\rho(dr),
\ee
where
\ce
\theta(u,r)&=&\sum_{i=1}^d\frac{\partial_i(\zeta^{1/2}(u)k(u))}{k(u)}\xi^i(r)\no\\
&=&\sum_{i=1}^d(\partial_i\log k(u)\zeta^{1/2}(u)+\frac12\zeta^{-1/2}(u)\partial_i\zeta(u))\xi^i(r).
\de

To establish the derivative formulae, we need to introduce some notations. First of all, for any $s\geq0$, let $\{K(t,s),t\geq s\}$ solve the following random ordinary differential equations on $\mR^m\times\mR^m$:
\ce
\frac{d}{dt}K(t,s)=(\nabla^{(1)}Z^{(1)})(X^{(1)}(t,x_1),X^{(2)}(t,x_2))K(t,s),\quad K(s,s)=I_{m\times m}.
\de
Denote
\ce
Q(t):=\int_0^t\phi(s)K(T,s)\nabla^{(2)}Z^{(1)}(X(s,x))B_0^TK^T(T,s)ds.
\de
For $t\in[0,T]$, let
\ce
\alpha^{(1)}(t):=K(t,0)h_1+\int_0^tK(t,s)\nabla^{(2)}Z^{(1)}(X(s,x))\alpha^{(2)}(s)ds
\de
and
\be\label{alpha 2}
\alpha^{(2)}(t)&:=&\frac{T-t}{T}h_2-\phi(t)B_0^TK^T(T,t)Q^{-1}(T)\int_0^T\frac{T-s}{T}K(T,s)\nabla^{(2)}Z^{(1)}(X(s,x))h_2ds\no\\
&&-\frac{\phi(t)B_0^TK^T(T,t)}{\int_0^T\xi^2(s)\,ds}\int_t^T\xi^2(s)Q^{-1}(s)K(T,0)h_1ds.
\ee
It is easy to see that
\be\label{alpha 1 variation}
\alpha^{(1)}(t)=h_1+\int_0^t\nabla Z^{(1)}(X(s,x))\alpha(s)ds,\quad t\in[0,T],
\ee
where $\alpha(\cdot)=(\alpha^{(1)}(\cdot),\alpha^{(2)}(\cdot))$. It also follows from the proof of Theorem 1.1 in \cite{WZ} that $\alpha^{(2)}(0)=h_2$, $\alpha^{(2)}(T)=0$ and $\alpha^{(1)}(T)=0$. Now we set
\ce
&&\tau^{\alpha}(t)\no\\
&:=&\frac{1}{\int_0^t\int_{\cO}\zeta(u)N(ds,du)}\<\int_0^t\int_{\cO}\int_R\zeta^{1/2}(u)\xi(r)N\odot\rho(ds,du,dr),\no\\
&&\sigma^{-1}(h_2-\alpha^{(2)}(t)+\int_0^t\nabla Z^{(2)}(X(s,x))\alpha(s)ds)\>.
\de
Then we have
\be\label{tau formula variation}
&&\<\int_0^t\int_{\cO}\int_R\zeta^{1/2}(u)\sigma\xi(r)N\odot\rho(ds,du,dr),\tau^{\alpha}(t)\>_{\mH}\no\\
&=&h_2-\alpha^{(2)}(t)+\int_0^t\nabla Z^{(2)}(X(s,x))\alpha(s)ds.
\ee

The Bismut weight appearing in Theorem \ref{derivative} is then explicitly given by
\ce
&&M(T,x_1,x_2,h)\\
&=&(\int_0^T\int_{\cO}N(dt,du))^{-1}\Bigl[-\sum_{i=1}^d\int_0^T\int_{\cO}\bigl[\zeta^{1/2}(u)(\partial_i\log k(u)\zeta^{1/2}(u)+\frac{1}{2}\zeta^{-1/2}(u)\partial_i\zeta(u))\\
&&+\zeta^{1/2}(u)\partial_i\zeta^{1/2}(u)\bigr]e_i^TN(dt,du)\Bigr]\sigma^{-1}(h_2-\alpha^{(2)}(T)+\int_0^T\nabla Z^{(2)}(X(t,x))\alpha(t)dt)\\
&&+(\int_0^T\int_{\cO}\zeta(u)N(dt,du))^{-2}\int_0^T\int_{\cO}\zeta(u)(\partial_1\zeta(u),\cdots,\partial_d\zeta(u))N(dt,du)\\
&&\sigma^{-1}(h_2-\alpha^{(2)}(T)+\int_0^T\nabla Z^{(2)}(X(t,x))\alpha(t)dt)\\
&&-(\int_0^T\int_{\cO}\zeta(u)N(dt,du))^{-1}\<\int_0^T\int_{\cO}\int_R\zeta^{1/2}(u)\xi^T(r)N\odot\rho(dt,du,dr),\\
&&\sigma^{-1}(\alpha^{(2)}(T))^{\sharp}\>_{\mH}\\
&&+(\int_0^T\int_{\cO}\zeta(u)N(dt,du))^{-1}\<\int_0^T\int_{\cO}\int_R\zeta^{1/2}(u)\xi^T(r)N\odot\rho(dt,du,dr),\\
&&\sigma^{-1}\int_0^T\nabla^{(2)}Z^{(2)}(X(t,x))X^{\sharp}(t,x)\alpha(t)dt\>_{\mH}\\
&&+(\int_0^T\int_{\cO}\zeta(u)N(dt,du))^{-1}\<\int_0^T\int_{\cO}\int_R\zeta^{1/2}(u)\xi^T(r)N\odot\rho(dt,du,dr),\\
&&\sigma^{-1}\int_0^T\nabla Z^{(2)}(X(t,x))\alpha^{\sharp}(t)dt\>_{\mH}.
\de

\subsection{Some auxiliary lemmas}

In order to prove Theorem \ref{derivative}, some useful lemmas will be used. The following Kunita's inequalities can be found in \cite[Lemma 2.3]{SZ} or \cite[Theorem 2.11]{K}.

\begin{lemma}\label{bdg}
\begin{enumerate}
\item
For any $p\geq1$, there is a positive constant $C$ such that
\ce
&&\mE\bigl[\sup_{t\in[0,T]}|\int_0^t\int_{\cO}\psi(s,u)N(ds,du)|^p\bigr]\\
&\leq&C\mE\bigl[\int_0^T\int_{\cO}|\psi(s,u)|\nu(du)ds\bigr]^p+C\mE\bigl[\int_0^T\int_{\cO}|\psi(s,u)|^p\nu(du)ds\bigr].
\de
\item
For any $p\geq2$, there is a positive constant $C$ such that
\ce
&&\mE\bigl[\sup_{t\in[0,T]}|\int_0^t\int_{\cO}\psi(s,u)\widetilde{N}(ds,du)|^p\bigr]\\
&\leq&C\mE\bigl[\int_0^T\int_{\cO}|\psi(s,u)|^2\nu(du)ds\bigr]^{p/2}+C\mE\bigl[\int_0^T\int_{\cO}|\psi(s,u)|^p\nu(du)ds\bigr].
\de
\end{enumerate}
\end{lemma}

The following estimate can be found in \cite[Lemma 2.5, Lemma 2.6]{WXZ1}.

\begin{lemma}\label{Xie}
Under the assumption \ref{assumption special}, we have
\begin{enumerate}
\item
for any $p\geq2$, there exists a positive constant $C$ such that for any $\varepsilon\in(0,1)$,
\ce
\int_{\{|u|\leq\varepsilon\}}|u|^p\nu(du)\leq C\varepsilon^{p-\alpha};
\de
\item
for any $p\geq1$, there exists a positive constant $C$ such that for any $t>0$ and $\varepsilon\in(0,1)$,
\ce
\mE\bigl[(\int_0^t\int_{\{|u|\leq\varepsilon\}}|u|^3N(ds,du))^{-p}\bigr]\leq C((t\varepsilon^{3-\alpha})^{-p}+t^{-\frac{3p}{\alpha}}).
\de
\end{enumerate}
\end{lemma}

The following lemma serves as the computational tool for divergence operator $\delta_{\sharp}$.

\begin{lemma}\label{divergence particular}
For any $t\in[0,T]$, we have
\ce
&&\delta_{\sharp}(\int_0^t\int_{\cO}\int_{R}\zeta^{1/2}(u)\xi^T(r)N\odot\rho(ds,du,dr))\\
&=&-\sum_{i=1}^d\int_0^t\int_{\cO}\bigl[\zeta^{1/2}(u)(\partial_i\log k(u)\zeta^{1/2}(u)+\frac{1}{2}\zeta^{-1/2}(u)\partial_i\zeta(u))\\
&&+\zeta^{1/2}(u)\partial_i\zeta^{1/2}(u)\bigr]e_i^TN(ds,du).
\de
\end{lemma}

\begin{proof}
First of all, we have
\ce
\varepsilon_{(s,u,r)}^+(\int_0^t\int_{\cO}\int_R\zeta^{1/2}(u')\xi^T(r')N\odot\rho(ds',du',dr'))=\zeta^{1/2}(u)\xi^T(r).
\de
This implies by (\ref{theta}) that
\ce
&&\widehat{\mE}\bigl[\delta_{\flat}(\varepsilon_{(s,u,r)}^+(\int_0^t\int_{\cO}\int_R\zeta^{1/2}(u')\xi^T(r')N\odot\rho(ds',du',dr')))(s,u)\bigr]\\
&=&-\int_R[\theta(u,r)\zeta^{1/2}(u)\xi^T(r)\\
&&+\zeta^{1/2}(u)\sum_{i=1}^d\partial_i\zeta^{1/2}(u)\xi^T(r)\xi^i(r)]\rho(dr)\\
&=&-\sum_{i=1}^d\bigl[\zeta^{1/2}(u)(\partial_i\log k(u)\zeta^{1/2}(u)+\frac{1}{2}\zeta^{-1/2}(u)\partial_i\zeta(u))\\
&&+\zeta^{1/2}(u)\partial_i\zeta^{1/2}(u)\bigr]e_i^T,
\de
where we have used
\ce
\int_R\xi^i(r)\xi^j(r)\rho(dr)=\delta_{ij},\quad\forall i,j=1,\cdots,d.
\de
Then by Proposition \ref{divergence}, we have
\ce
&&\delta_{\sharp}[\int_0^t\int_{\cO}\int_R\zeta^{1/2}(u)\xi^T(r)N\odot\rho(ds,du,dr)]\\
&=&\int_0^t\int_{\cO}\varepsilon_{(s,u)}^-(\widehat{\mE}\bigl[\delta_{\flat}(\varepsilon_{(s,u,r)}^+(\int_0^t\int_{\cO}\int_R\zeta^{1/2}(u')\xi^T(r')\\
&&N\odot\rho(ds',du',dr')))(s,u)\bigr])N(ds,du)\\
&=&-\sum_{i=1}^d\int_0^t\int_{\cO}\bigl[\zeta^{1/2}(u)(\partial_i\log k(u)\zeta^{1/2}(u)+\frac{1}{2}\zeta^{-1/2}(u)\partial_i\zeta(u))\\
&&+\zeta^{1/2}(u)\partial_i\zeta^{1/2}(u)\bigr]e_i^TN(ds,du).
\de
Therefore the proof is complete.
\end{proof}

The following five lemmas give some estimates on the divergence operator $M(T,x_1,x_2,h)$. 

\bl\label{k1}
Under the assumption of Theorem \ref{derivative}, for any $q>1$,
\ce
\mE\bigl[\sup_{t\in[0,T]}\|X^{\sharp}(t,x)\|_{\mH}^q\bigr]\leq CT^{q/2}\varepsilon^{(3-\alpha)q/2}+CT\varepsilon^{3q/2-\alpha}.
\de
\el
\begin{proof}
Since
\ce
X^{\sharp}(t,x)&=&\int_0^t\nabla Z(X(s,x))X^{\sharp}(s,x)ds\\
&&+(0,\int_0^t\int_{\cO}\int_R\zeta^{1/2}(u)\sigma\xi(r)N\odot\rho(ds,du,dr)),
\de
by the assumption \ref{coefficient}, we have
\ce
\|X^{\sharp}(t,x)\|_{\mH}\leq C\int_0^t\|X^{\sharp}(s,x)\|_{\mH}ds+\|\sigma\|(\int_0^t\int_{\cO}\zeta(u)N(ds,du))^{1/2}.
\de 
Now in view of Gronwall's inequality, there exists a constant $C>0$ such that
\ce
\|X^{\sharp}(t,x)\|_{\mH}\leq C\|\sigma\|(\int_0^T\int_{\cO}\zeta(u)N(ds,du))^{1/2},
\de
and this implies that
\ce
\mE\bigl[\sup_{t\in[0,T]}\|X^{\sharp}(t,x)\|_{\mH}^q\bigr]\leq C\|\sigma\|\mE\bigl[(\int_0^T\int_{\cO}\zeta(u)N(ds,du))^{q/2}\bigr].
\de
Then using Lemma \ref{Xie}, it follows that
\ce
\mE\bigl[\sup_{t\in[0,T]}\|X^{\sharp}(t,x)\|_{\mH}^q\bigr]\leq CT^{q/2}\varepsilon^{(3-\alpha)q/2}+CT\varepsilon^{3q/2-\alpha}.
\de 
Therefore we complete the proof.
\end{proof}

\bl\label{k2}
Under the assumption of Theorem \ref{derivative}, for any $q>1$,
\ce
\mE\bigl[\sup_{t\in[0,T]}\|K^{\sharp}(T,t)\|_{\mH}^q\bigr]\leq CT^{3q/2}\varepsilon^{(3-\alpha)q/2}+CT^{1+q}\varepsilon^{3q/2-\alpha}.
\de
\el

\begin{proof}
First of all, it is obvious that
\ce
\frac{d}{dt}K^{\sharp}(t,s)=(\nabla\nabla^{(1)}Z^{(1)})(X(t,x))X^{\sharp}(t,x)+(\nabla^{(1)}Z^{(1)})(X(t,x))K^{\sharp}(t,s).
\de
We first record that the first variation process is uniformly bounded. Indeed, $K(t,s)$ solves
\ce
\frac{d}{dt}K(t,s)=(\nabla^{(1)}Z^{(1)})(X(t,x))K(t,s),\qquad K(s,s)=I_{m\times m}.
\de
Since $\|\nabla^{(1)}Z^{(1)}\|\le\|\nabla Z\|\le C$ by the assumption \ref{coefficient}, Gronwall's inequality gives, uniformly for $0\le s\le t\le T$,
\ce
\|K(t,s)\|\le\exp\!\Big(\int_s^t\|\nabla^{(1)}Z^{(1)}(X(r,x))\|\,dr\Big)\le e^{C(t-s)}\le e^{CT}=:C_K.
\de
The assumption \ref{coefficient} also gives $\|\nabla\nabla^{(1)}Z^{(1)}\|\le\|\nabla^2 Z^{(1)}\|\le C$. Hence, taking the $\mH$-norm in the equation for $K^\sharp$ and using $\|K(r,s)\|\le C_K$, we obtain
\ce
\|K^{\sharp}(t,s)\|_{\mH}&\leq& C_K\int_s^t\|\nabla\nabla^{(1)}Z^{(1)}(X(r,x))\|\cdot\|X^{\sharp}(r,x)\|_{\mH}\,dr\\
&&+\int_s^t\|\nabla^{(1)}Z^{(1)}(X(r,x))\|\cdot\|K^{\sharp}(r,s)\|_{\mH}\,dr\\
&\leq&C\int_s^t\|X^{\sharp}(r,x)\|_{\mH}\,dr+\int_s^t\|K^{\sharp}(r,s)\|_{\mH}\,dr.
\de
Then the desired estimate can be obtained by Gronwall's inequality and Lemma \ref{k1}.
\end{proof}

\bl\label{k3}
Under the assumption of Theorem \ref{derivative}, $Q(t)$ is invertible for $t\in(0,T]$ with
\ce
\|Q^{-1}(t)\|\leq\frac{1}{(1-\varepsilon)\xi(t)},\quad t\in(0,T].
\de
Moreover, for any $q>1$, we also have
\ce
\mE\bigl[\|(Q(t))^{\sharp}\|_{\mH}^q\bigr]\leq C(T^{3q/2}+T^{5q/2})\varepsilon^{(3-\alpha)q/2}+C(T^{q+1}+T^{2q+1})\varepsilon^{3q/2-\alpha}.
\de
\el

\begin{proof}
The estimate 
\ce
\|Q^{-1}(t)\|\le \frac{1}{(1-\varepsilon)\xi(t)},\qquad t\in(0,T],
\de
follows directly from the proof of Theorem 1.1 in \cite{WZ}, and hence we omit the details.

We now estimate
\ce
\mE\bigl[\|(Q(t))^\sharp\|_{\mH}^q\bigr].
\de
From the definition
\ce
Q(t)=\int_0^t \phi(s)K(T,s)\nabla^{(2)}Z^{(1)}(X(s,x))B_0^TK^T(T,s)ds,
\de
taking the Malliavin derivative (with respect to the $\sharp$-operator) and using the chain rule, we obtain
\ce
(Q(t))^\sharp
&=& \int_0^t \phi(s) [
K^\sharp(T,s)\nabla^{(2)}Z^{(1)}(X(s,x))B_0^TK^T(T,s) \\
&& + K(T,s)(\nabla^{(2)}Z^{(1)}(X(s,x)))^\sharp B_0^TK^T(T,s) \\
&& + K(T,s)\nabla^{(2)}Z^{(1)}(X(s,x))B_0^T(K^T(T,s))^\sharp
] ds\\
&=& \int_0^t \phi(s) [
K^\sharp(T,s)\nabla^{(2)}Z^{(1)}(X(s,x))B_0^TK^T(T,s) \\
&& + K(T,s)\nabla\nabla^{(2)}Z^{(1)}(X(s,x))X^\sharp(s,x)B_0^TK^T(T,s) \\
&& + K(T,s)\nabla^{(2)}Z^{(1)}(X(s,x))B_0^T(K^T(T,s))^\sharp
] ds.
\de
To bound the three terms in the preceding display, we collect the needed estimates. First of all, $\phi$ is bounded on $[0,T]$ and $B_0$ is a constant matrix. Secondly, the assumption \ref{coefficient} gives the boundedness of both $\nabla^{(2)}Z^{(1)}$ and its derivative $\nabla\nabla^{(2)}Z^{(1)}$, since
\ce
\|\nabla^{(2)}Z^{(1)}\|\le\|\nabla Z\|\le C,\qquad
\|\nabla\nabla^{(2)}Z^{(1)}\|\le\|\nabla^2 Z^{(1)}\|\le C.
\de
Thirdly, the first variation process and its transpose are uniformly bounded on $0\le s\le T$; indeed, the first-variation equation and Gronwall's inequality (as in the proof of Lemma \ref{k2}) yield
\ce
\|K(T,s)\|=\|K^T(T,s)\|\le e^{CT}=:C_K.
\de
Applying these bounds term by term in the expression for $(Q(t))^\sharp$, and using the submultiplicativity of the operator norm (the first and third terms leave the sharp factors of $K^T(T,s)$ and $K(T,s)$, respectively, while the middle one leaves $X^\sharp(s,x)$), there exists a constant $C>0$ such that
\ce
\|(Q(t))^\sharp\|_{\mH}
\le C\int_0^t\bigl(
\|(K(T,s))^\sharp\|_{\mH}
+\|X^\sharp(s,x)\|_{\mH}
+\|(K^T(T,s))^\sharp\|_{\mH}
\bigr)\,ds.
\de
Noting that $\|(K^T(T,s))^\sharp\|_{\mH}=\|(K(T,s))^\sharp\|_{\mH}$, we get
\ce
\|(Q(t))^\sharp\|_{\mH}
\le C\int_0^T \|X^\sharp(s,x)\|_{\mH}ds
+C\int_0^T \|(K(T,s))^\sharp\|_{\mH}ds.
\de
Taking the $q$-th power and expectation, and applying H\"older's inequality, yields
\ce
\mE\bigl[\|(Q(t))^\sharp\|_{\mH}^q\bigr]
&\le& C T^{q-1}\int_0^T \mE\bigl[\|X^\sharp(s,x)\|_{\mH}^q\bigr]ds
+C T^{q-1}\int_0^T \mE\bigl[\|(K(T,s))^\sharp\|_{\mH}^q\bigr]ds \\
&\le& C T^q \sup_{s\in[0,T]}\mE\bigl[\|X^\sharp(s,x)\|_{\mH}^q\bigr]
+C T^q \sup_{s\in[0,T]}\mE\bigl[\|(K(T,s))^\sharp\|_{\mH}^q\bigr].
\de
Then by Lemma \ref{k1} and Lemma \ref{k2}, we obtain
\ce
\mE\bigl[\|(Q(t))^\sharp\|_{\mH}^q\bigr]
\le C(T^{3q/2}+T^{5q/2})\varepsilon^{(3-\alpha)q/2}+C(T^{q+1}+T^{2q+1})\varepsilon^{3q/2-\alpha}
\de
for some constant $C>0$. This completes the proof.
\end{proof}

\begin{lemma}\label{k4}
Under the assumption of Theorem \ref{derivative}, for any $q>1$, we have
\ce
&&\mE\bigl[|\sum_{i=1}^d\int_0^T\int_{\cO}\bigl[\zeta^{1/2}(u)(\partial_i\log k(u)\zeta^{1/2}(u)+\frac{1}{2}\zeta^{-1/2}(u)\partial_i\zeta(u))\\
&&+\zeta^{1/2}(u)\partial_i\zeta^{1/2}(u)\bigr]e_i^TN(dt,du)|^q\bigr]\\
&\leq&C(T\varepsilon^{2-\alpha})^q+CT\varepsilon^{2q-\alpha}
\de
and
\ce
&&\mE\bigl[|\int_0^T\int_{\cO}\zeta(u)(\partial_1\zeta(u),\cdots,\partial_d\zeta(u))N(dt,du)|^q\bigr]\\
&\leq&CT\varepsilon^{5q-\alpha}+CT^q\varepsilon^{5q-\alpha q}.
\de
\end{lemma}

\begin{proof}
Applying Lemma \ref{bdg} and Lemma \ref{Xie}, we have
\ce
&&\mE\bigl[|\int_0^T\int_{\cO}\bigl[\zeta^{1/2}(u)(\partial_i\log k(u)\zeta^{1/2}(u)+\frac{1}{2}\zeta^{-1/2}(u)\partial_i\zeta(u))\\
&&+\zeta^{1/2}(u)\partial_i\zeta^{1/2}(u)\bigr]e_i^TN(dt,du)|^q\bigr]\\
&\leq&\mE\bigl[\int_0^T\int_{\cO}[\zeta^{1/2}(u)(\partial_i\log k(u)\zeta^{1/2}(u)+\frac{1}{2}\zeta^{-1/2}(u)\partial_i\zeta(u))+\zeta^{1/2}(u)\partial_i\zeta^{1/2}(u)]\nu(du)dt\bigr]^q\\
&&+\mE\bigl[\int_0^T\int_{\cO}[\zeta^{1/2}(u)(\partial_i\log k(u)\zeta^{1/2}(u)+\frac{1}{2}\zeta^{-1/2}(u)\partial_i\zeta(u))+\zeta^{1/2}(u)\partial_i\zeta^{1/2}(u)]^q\nu(du)dt\bigr]\\
&\leq&C(T\varepsilon^{2-\alpha})^q+CT\varepsilon^{2q-\alpha}.
\de
For the second estimate, using Lemma \ref{bdg} and Lemma \ref{Xie} again, it follows that
\ce
&&\mE\bigl[|\int_0^T\int_{\cO}\zeta(u)(\partial_1\zeta(u),\cdots,\partial_d\zeta(u))N(dt,du)|^q\bigr]\\
&\leq&C\mE\bigl[\int_0^T\int_{\cO}|\zeta(u)(\partial_1\zeta(u),\cdots,\partial_d\zeta(u))|^q\nu(du)dt\bigr]\\
&&+C\mE\bigl[\int_0^T\int_{\cO}|\zeta(u)(\partial_1\zeta(u),\cdots,\partial_d\zeta(u))|\nu(du)dt\bigr]^q\\
&\leq&CT\varepsilon^{5q-\alpha}+CT^q\varepsilon^{5q-\alpha q}.
\de
Then the desired estimate is obtained.
\end{proof}

\begin{lemma}\label{k5}
Under the assumption of Theorem \ref{derivative}, for any $q>1$, there exists a constant $C>0$ such that
\be\label{k5a}
&&\mE\bigl[\sup_{0\le t\le T}\|(\alpha^{(2)}(t))^\sharp\|_{\mH}^q\bigr]\no\\
&\le& C\bigl(|h_1|^q+|h_2|^q\bigr)
\biggl[
\frac{1}{\xi^{q}(T)}
\bigl[ (T^{3q/2}+T^{5q/2})\varepsilon^{(3-\alpha)q/2}
      + (T^{q+1}+T^{2q+1})\varepsilon^{3q/2-\alpha} \bigr] \no\\
&& + \frac{T^q}{\xi^{2q}(T)}
\bigl[ (T^{3q/2}+T^{5q/2})\varepsilon^{(3-\alpha)q/2}
      + (T^{q+1}+T^{2q+1})\varepsilon^{3q/2-\alpha} \bigr] \no\\
&& + \frac{1}{\bigl(\int_0^T \xi^2(s)\,ds\bigr)^q}
\Bigl(
\bigl( \int_0^T \xi(s)\,ds \bigr)^q
\bigl( T^{3q/2}\varepsilon^{(3-\alpha)q/2}
      + T^{q+1}\varepsilon^{3q/2-\alpha} \bigr) \no\\
&& + T^q
\bigl[ (T^{3q/2}+T^{5q/2})\varepsilon^{(3-\alpha)q/2}
      + (T^{q+1}+T^{2q+1})\varepsilon^{3q/2-\alpha} \bigr]
\Bigr)\biggr],
\ee
\be\label{k5c}
&&\mE\bigl[\|\int_0^T \nabla Z^{(2)}(X(t,x))\alpha^\sharp(t)dt\|_{\mH}^q\bigr] \no\\
&\le& C\bigl(|h_1|^q+|h_2|^q\bigr)
\biggl[
\frac{1}{\xi^{q}(T)}
\bigl[ (T^{7q/2}+T^{9q/2})\varepsilon^{(3-\alpha)q/2}
      + (T^{3q+1}+T^{4q+1})\varepsilon^{3q/2-\alpha} \bigr] \no\\
&& + \frac{T^{3q}}{\xi^{2q}(T)}
\bigl[ (T^{3q/2}+T^{5q/2})\varepsilon^{(3-\alpha)q/2}
      + (T^{q+1}+T^{2q+1})\varepsilon^{3q/2-\alpha} \bigr] \no\\
&& + \frac{T^{2q}}{\bigl(\int_0^T \xi^2(s)ds\bigr)^q}
\Bigl[
\bigl( \int_0^T \xi(s)\,ds \bigr)^q
\bigl( T^{3q/2}\varepsilon^{(3-\alpha)q/2}
      + T^{q+1}\varepsilon^{3q/2-\alpha} \bigr) \no\\
&& + T^q
\bigl[ (T^{3q/2}+T^{5q/2})\varepsilon^{(3-\alpha)q/2}
      + (T^{q+1}+T^{2q+1})\varepsilon^{3q/2-\alpha} \bigr]
\Bigr] \no\\
&& + T^{5q/2}\varepsilon^{(3-\alpha)q/2}
      + T^{7q/2}\varepsilon^{(3-\alpha)q/2}
      + T^{2q+1}\varepsilon^{3q/2-\alpha}
      + T^{3q+1}\varepsilon^{3q/2-\alpha}
\biggr],
\ee
and
\be\label{k5d}
&&\mE\bigl[\|\int_0^T \nabla^{(2)}Z^{(2)}(X(t,x))X^\sharp(t,x)\alpha(t)dt\|_{\mH}^q\bigr] \notag\\
        &\le& C \bigl(|h_1|^q+|h_2|^q\bigr)\Lambda^q(T)\bigl( T^{3q/2}\varepsilon^{(3-\alpha)q/2} + T^{q+1}\varepsilon^{3q/2-\alpha} \bigr).
\ee
\end{lemma}

\begin{proof}

We prove the three estimates separately.

(1) Recall the definition of $\alpha^{(2)}(t)$ from (\ref{alpha 2}):
\ce
\alpha^{(2)}(t)&=&J(t)+I(t),
\de
where
\ce
J(t):=\frac{T-t}{T}h_2
-\phi(t)B_0^T K^T(T,t)Q^{-1}(T)\int_0^T \frac{T-s}{T}K(T,s)\nabla^{(2)}Z^{(1)}(X(s,x))h_2ds
\de
and
\ce
I(t):=-\frac{\phi(t)B_0^T K^T(T,t)}{\int_0^T \xi^2(s)ds}\int_t^T \xi^2(s)Q^{-1}(s)K(T,0)h_1ds.
\de
We first estimate $J^\sharp(t)$. Differentiating $J(t)$ with respect to the $\sharp$-operator gives
\ce
&&J^\sharp(t)\\
&=& -\phi(t)B_0^T[
(K^T(T,t))^\sharp Q^{-1}(T)C(T)\\
&&+ K^T(T,t)(Q^{-1}(T))^\sharp C(T)\\
&&+ K^T(T,t)Q^{-1}(T)C^\sharp(T)
],
\de
where
\ce
C(t):=\int_0^T \frac{T-s}{T}K(T,s)\nabla^{(2)}Z^{(1)}(X(s,x))h_2ds.
\de
Using the uniform bounds $\|K(T,s)\|\le C$, $\|Q^{-1}(T)\|\le C/\xi(T)$, and $\|(Q^{-1}(T))^\sharp\|_{\mH}\le C\|Q^\sharp(T)\|_{\mH}/\xi^2(T)$, we obtain
\ce
\|J^\sharp(t)\|_{\mH}
\le C|h_2|[
\frac{T\|K^\sharp(T,t)\|_{\mH}}{\xi(T)}
+ \frac{T\|Q^\sharp(T)\|_{\mH}}{\xi^2(T)}
+ \frac{1}{\xi(T)}\int_0^T (\|K^\sharp(T,s)\|_{\mH}+\|X^\sharp(s,x)\|_{\mH})ds
].
\de
Taking the supremum over $0\le t\le T$ and then the $q$-th moment, applying H\"older's inequality and Lemmas \ref{k1}--\ref{k3}, yields
\be\label{A}
&&\mE\bigl[\sup_{0\le t\le T}\|J^\sharp(t)\|_{\mH}^q\bigr]\no\\
&\le& C|h_2|^q \Bigl[
\frac{T^q}{\xi^q(T)}\mE\bigl[\sup_{0\le s\le T}\|K^\sharp(T,s)\|_{\mH}^q\bigr]
+ \frac{T^q}{\xi^{2q}(T)}\mE\bigl[\|Q^\sharp(T)\|_{\mH}^q\bigr] \no\\
&& + \frac{T^{q-1}}{\xi^q(T)}\int_0^T \bigl(\mE[\|K^\sharp(T,s)\|_{\mH}^q]+\mE\bigl[\|X^\sharp(s,x)\|_{\mH}^q\bigr]\bigr)ds
\Bigr] \no\\
&\le& C|h_2|^q \Bigl[
\frac{T^q}{\xi^{q}(T)}
\bigl(T^{3q/2}\varepsilon^{(3-\alpha)q/2}+T^{q+1}\varepsilon^{3q/2-\alpha}\bigr) \no\\
&& + \frac{T^q}{\xi^{2q}(T)}
\bigl[(T^{3q/2}+T^{5q/2})\varepsilon^{(3-\alpha)q/2}+(T^{q+1}+T^{2q+1})\varepsilon^{3q/2-\alpha}\bigr] \no\\
&& + \frac{T^q}{\xi^{q}(T)}
\bigl[(T^{q/2}+T^{3q/2})\varepsilon^{(3-\alpha)q/2}+(T+T^{q+1})\varepsilon^{3q/2-\alpha}\bigr]
\Bigr].
\ee
Now we deal with the term $I(t)$. Differentiating $I(t)$ gives
\ce
I^\sharp(t)
= -\frac{\phi(t)}{\int_0^T \xi^2(s)ds}B_0^T\bigl[
(K^T(T,t))^\sharp G(t) + K^T(T,t)G^\sharp(t)
\bigr],
\de
where
\ce
G(t)=\int_t^T \xi^2(s)Q^{-1}(s)K(T,0)h_1ds,
\de
and
\ce
G^\sharp(t)=\int_t^T \xi^2(s)\bigl[(Q^{-1}(s))^\sharp K(T,0) + Q^{-1}(s)K^\sharp(T,0)\bigr]h_1ds.
\de
Using $\|Q^{-1}(s)\|_{\mH}\le C/\xi(s)$ and $\|(Q^{-1}(s))^\sharp\|_{\mH}\le C\|Q^\sharp(s)\|_{\mH}/\xi^2(s)$, we get
\ce
\|G(t)\|\le C|h_1|\int_t^T \xi(s)ds \le C|h_1|\int_0^T \xi(s)ds
\de
and
\ce
\|G^\sharp(t)\|_{\mH}
\le C|h_1|\bigl[
\int_0^T \|Q^\sharp(s)\|_{\mH}ds
+ \|K^\sharp(T,0)\|_{\mH}\int_0^T \xi(s)ds
\bigr].
\de
Also, $\|(K^T(T,t))^\sharp\|_{\mH}=\|K^\sharp(T,t)\|_{\mH}$. Hence
\ce
\|I^\sharp(t)\|_{\mH}
\le \frac{C|h_1|}{\int_0^T \xi^2(s)ds}
\bigl[
\|K^\sharp(T,t)\|_{\mH}\int_0^T \xi(s)ds
+ \int_0^T \|Q^\sharp(s)\|_{\mH}ds
+ \|K^\sharp(T,0)\|_{\mH}\int_0^T \xi(s)ds
\bigr].
\de
Taking the supremum over $0\le t\le T$ and then the $q$-th moment gives
\ce
&&\mE\bigl[\sup_{t\le T}\|I^\sharp(t)\|_{\mH}^q\bigr] \\
&\le& C|h_1|^q \frac{1}{(\int_0^T \xi^2(s)ds)^q}
\bigl[
(\int_0^T \xi(s)\,ds)^q \mE\bigl[\sup_{s\le T}\|K^\sharp(T,s)\|_{\mH}^q\bigr] \\
&& + T^{q-1}\int_0^T \mE\bigl[\|Q^\sharp(s)\|_{\mH}^q\bigr]ds
+ (\int_0^T \xi(s)ds)^q \mE\bigl[\|K^\sharp(T,0)\|_{\mH}^q\bigr]
\bigr].
\de
Applying Lemmas \ref{k1}--\ref{k3} yields
\be\label{B}
&&\mE\bigl[\sup_{t\le T}\|I^\sharp(t)\|_{\mH}^q\bigr] \no\\
&\le& C|h_1|^q \frac{1}{\bigl(\int_0^T \xi^2(s)ds\bigr)^q}
\Bigl[
\bigl(\int_0^T \xi(s)\,ds\bigr)^q
\bigl(T^{3q/2}\varepsilon^{(3-\alpha)q/2}+T^{q+1}\varepsilon^{3q/2-\alpha}\bigr) \no\\
&& + T^q
\bigl[(T^{3q/2}+T^{5q/2})\varepsilon^{(3-\alpha)q/2}+(T^{q+1}+T^{2q+1})\varepsilon^{3q/2-\alpha}\bigr]
\Bigr].
\ee
Finally, combining (\ref{A}) and (\ref{B}) gives (\ref{k5a}).

(2) For (\ref{k5c}), from the definition
\ce
\alpha^{(1)}(t)=K(t,0)h_1+\int_0^t K(t,s)\nabla^{(2)}Z^{(1)}(X(s,x))\alpha^{(2)}(s)ds,
\de
differentiating with respect to the $\sharp$-operator and applying Gronwall's inequality yields
\ce
\|\alpha^\sharp(t)\|_{\mH}
&\le& C(|h_1|+|h_2|)
[
\|K^\sharp(t,0)\|_{\mH}
+\int_0^t \|K^\sharp(t,s)\|_{\mH}ds\\
&&+\int_0^t \|X^\sharp(s,x)\|_{\mH}ds]
+\int_0^t \|(\alpha^{(2)}(s))^\sharp\|_{\mH}ds\\
&=:&C(|h_1|+|h_2|) (A_1(t)+A_2(t)+A_3(t))+A_4(t).
\de
Taking the $q$-th moment and using Minkowski's inequality, we obtain
\ce
\mE\bigl[(\int_0^T\|\alpha^\sharp(t)\|_{\mH}dt)^q\bigr]
\le C(|h_1|^q+|h_2|^q) \sum_{i=1}^3 \mE\bigl[(\int_0^T A_i(t)dt)^q\bigr]+C\mE\bigl[(\int_0^TA_4(t)dt)^q\bigr].
\de
We estimate each term individually. For $A_1(t)=\|K^\sharp(t,0)\|_{\mH}$, using Lemma \ref{k2} (with $s=0$) gives
\ce
&&\mE\bigl[(\int_0^T A_1(t)dt)^q\bigr]\\
&\le& C T^q \sup_{t\le T}\mE\bigl[\|K^\sharp(t,0)\|^q\bigr]\\
&\le& C(T^{5q/2}\varepsilon^{(3-\alpha)q/2}+T^{2q+1}\varepsilon^{3q/2-\alpha}).
\de
For $A_2(t)=\int_0^t \|K^\sharp(t,s)\|_{\mH}ds$, by Fubini theorem and Lemma \ref{k2},
\ce
\mE\bigl[(\int_0^T A_2(t)dt)^q\bigr]
\le C T^{2q}\sup_{0\le s\le t\le T}\mE\bigl[\|K^\sharp(t,s)\|_{\mH}^q\bigr]
\le C(T^{7q/2}\varepsilon^{(3-\alpha)q/2}+T^{3q+1}\varepsilon^{3q/2-\alpha}).
\de
For $A_3(t)=\int_0^t \|X^\sharp(s,x)\|_{\mH}ds$, using Lemma \ref{k1},
\ce
\mE\bigl[(\int_0^T A_3(t)dt)^q\bigr]
\le C T^{2q} \sup_{0\le s\le T}\mE\bigl[\|X^\sharp(s,x)\|_{\mH}^q\bigr]
\le C(T^{5q/2}\varepsilon^{(3-\alpha)q/2}+T^{2q+1}\varepsilon^{3q/2-\alpha}).
\de
For $A_4(t)=\int_0^t \|(\alpha^{(2)}(s))^\sharp\|_{\mH}ds$, we use the bound obtained in (\ref{k5a}). More precisely,
\ce
\mE\bigl[(\int_0^T A_4(t)dt)^q\bigr]
\le C T^{2q} \sup_{0\le s\le T}\mE\bigl[\|(\alpha^{(2)}(s))^\sharp\|_{\mH}^q\bigr].
\de
Substituting the estimate (\ref{k5a}) for $\|(\alpha^{(2)}(s))^\sharp\|_{\mH}^q$, using $\xi(s)\le \xi(T)$ and $\int_0^s \xi^2(r)dr \le \int_0^T \xi^2(r)dr$, and integrating over $s\in[0,T]$, we obtain, after collecting powers of $T$ and merging constants, exactly the first three terms on the right-hand side of (\ref{k5c}) (with the corresponding $T$-powers). Combining all four estimates and noting the boundedness of $\nabla Z^{(2)}$, we arrive at the desired result (\ref{k5c}).

(3) Now we deal with the estimate (\ref{k5d}). From the explicit expression of $\alpha^{(2)}(t)$ (see (\ref{alpha 2})) and the bound $\|Q^{-1}(t)\|\le C/\xi(t)$, we obtain
\ce
|\alpha^{(2)}(t)| \le C|h_2|(1+\frac{1}{\xi(T)})
+ C|h_1|\frac{\int_0^T \xi(s)ds}{\int_0^T \xi^2(s)ds}.
\de
Together with $\alpha^{(1)}(t)=K(t,0)h_1+\int_0^t K(t,s)\nabla^{(2)}Z^{(1)}(X(s,x))\alpha^{(2)}(s)ds$ and the uniform boundedness of $K$, this yields
\be\label{alpha}
|\alpha(t)| \le C (|h_1|+|h_2|) (1 + \frac{1}{\xi(T)} + \frac{\int_0^T \xi(s)ds}{\int_0^T \xi^2(s)ds})=C (|h_1|+|h_2|)\Lambda(T). 
\ee
Since $\nabla^{(2)}Z^{(2)}$ is bounded, using Lemma \ref{k1}, H\"older's inequality and (\ref{alpha}), we get
\ce
&&\mE\bigl[\|\int_0^T \nabla^{(2)}Z^{(2)}(X(t,x))X^\sharp(t,x)\alpha(t)dt\|_{\mH}^q\bigr]\\ 
&\le& C \mE\bigl[ ( \int_0^T \|X^\sharp(t,x)\|_{\mH} \cdot |\alpha(t)| dt )^q \bigr] \\
&\le& C (|h_1|^q+|h_2|^q) \Lambda^q(T) \mE\bigl[ ( \int_0^T \|X^\sharp(t,x)\|_{\mH} dt )^q \bigr] \\
&\le& C (|h_1|^q+|h_2|^q) \Lambda^q(T) T^{q-1} \int_0^T \mE\bigl[\|X^\sharp(t,x)\|_{\mH}^q\bigr] dt \\
&\le& C (|h_1|^q+|h_2|^q) \Lambda^q(T)
( T^{3q/2}\varepsilon^{(3-\alpha)q/2} + T^{q+1}\varepsilon^{3q/2-\alpha} ),
\de
and this is desired result.

Thus all three estimates are established.
\end{proof}

The following lemma is crucial to the proof of Theorem \ref{derivative}.

\begin{lemma}\label{divergence expression}
Under the assumption of Theorem \ref{derivative}, $\tau^{\alpha}(T)\in\text{dom}(\delta_{\sharp})$, and for any $p>1$, there exists a positive constant $C$ such that for any $T\in(0,1)$, $h=(h_1,h_2)\in\mR^m\times\mR^d$,
\ce
&&\mE\bigl[|\delta_{\sharp}(\tau^{\alpha}(T))|^p\bigr]\\
&\le& C_p \bigl(|h_1|^p+|h_2|^p\bigr)\Bigl[
\Lambda^p(T)T^{-p/\alpha}+ \Lambda^p(T)\\
&&+ \frac{T^{4p}}{\xi^{p}(T)} + \frac{T^{5p}}{\xi^{2p}(T)}\\
&&+ \frac{T^{6p}}{\bigl(\int_0^T\xi^2(s)ds\bigr)^p}+ \frac{\bigl(\int_0^T\xi(s)ds\bigr)^p T^{4p}}{\bigl(\int_0^T\xi^2(s)ds\bigr)^p}\\
&&+ T^{2p} + T^{3p}
\Bigr].
\de
Moreover, we also have
\ce
\delta_{\sharp}(\tau^{\alpha}(T))=M(T,x_1,x_2,h).
\de
\end{lemma}

\begin{proof}
(1) First of all, we show that for any $p>1$, there exists a positive constant $C$ such that for any $T\in(0,1)$, $h=(h_1,h_2)\in\mR^m\times\mR^d$,
\ce
&&\mE\bigl[|M(T,x_1,x_2,h)|^p\bigr]\\
&\le& C_p \bigl(|h_1|^p+|h_2|^p\bigr)\Bigl[
\Lambda^p(T)T^{-p/\alpha}+ \Lambda^p(T)\\
&&+ \frac{T^{4p}}{\xi^{p}(T)} + \frac{T^{5p}}{\xi^{2p}(T)}\\
&&+ \frac{T^{6p}}{\bigl(\int_0^T\xi^2(s)ds\bigr)^p}+ \frac{\bigl(\int_0^T\xi(s)ds\bigr)^p T^{4p}}{\bigl(\int_0^T\xi^2(s)ds\bigr)^p}\\
&&+ T^{2p} + T^{3p}
\Bigr].
\de
For simplicity, from now on, we define $|h|^p=|h_1|^p+|h_2|^p$.

{\bf Step 1:}

From the definition,
\ce
M(T,x_1,x_2,h)
&=& I^{-1} A_0 B
 + I^{-2} A_1 B
 - I^{-1}\langle A_2, \sigma^{-1}(\alpha^{(2)}(T))^{\sharp}\rangle_{\mH} \\
&& + I^{-1}\langle A_2, \sigma^{-1}\int_0^T \nabla^{(2)}Z^{(2)}(X(t,x))X^{\sharp}(t,x)\alpha(t)dt\rangle_{\mH} \\
&& + I^{-1}\langle A_2, \sigma^{-1}\int_0^T \nabla Z^{(2)}(X(t,x))\alpha^{\sharp}(t)dt\rangle_{\mH}\\
&=:&M_1+M_2+M_3+M_4+M_5,
\de
where
\ce
I &:=&\int_0^T\int_{\cO}\zeta(u)N(dt,du),\\
B &:=&\sigma^{-1}(h_2-\alpha^{(2)}(T)+\int_0^T\nabla Z^{(2)}(X(t,x))\alpha(t)dt),\\
A_0 &:=&
-\sum_{i=1}^d\int_0^T\int_{\cO}
(\zeta^{1/2}(u)\partial_i\log k(u)\zeta^{1/2}(u)+\tfrac12\zeta^{-1/2}(u)\partial_i\zeta(u)
+\zeta^{1/2}(u)\partial_i\zeta^{1/2}(u))e_i^T N(dt,du),\\
A_1 &:=&\int_0^T\int_{\cO}\zeta(u)(\partial_1\zeta(u),\dots,\partial_d\zeta(u))N(dt,du),\\
A_2 &:=&\int_0^T\int_{\cO}\int_R \zeta^{1/2}(u)\xi^T(r)N\odot\rho(dt,du,dr).
\de
Since $\alpha^{(2)}(T)=0$ and $\nabla Z^{(2)}$ is bounded by the assumption \ref{coefficient}, we have
\ce
|B| \le C ( |h_2| + \int_0^T |\alpha(t)|dt ).
\de
Therefore using (\ref{alpha}), we have
\be\label{1}
|B| \le C |h|\Lambda(T).
\ee

{\bf Step 2:}

The following estimates are consequences of Lemmas \ref{k1}-\ref{k5} (with constants independent of $T,\varepsilon,h$):

\be
N_q &:=& \mE\bigl[I^{-q}\bigr]
      \le C((T\varepsilon^{3-\alpha})^{-q}+T^{-3q/\alpha}), \label{2}\\
A_q^{(0)} &:=& \mE\bigl[|A_0|^q\bigr]
           \le C(T^q\varepsilon^{(2-\alpha)q}+T\varepsilon^{2q-\alpha}), \label{3}\\
A_q^{(1)} &:=& \mE\bigl[|A_1|^q\bigr]
           \le C(T^q\varepsilon^{q(5-\alpha)}+T\varepsilon^{5q-\alpha}), \label{4}\\
S_q &:=& \mE\bigl[\|A_2\|_{\mH}^q\bigr]
      \le C(T^{q/2}\varepsilon^{(3-\alpha)q/2}+T\varepsilon^{3q/2-\alpha}), \label{5}\\
R_q &:=& \mE\bigl[\|(\alpha^{(2)}(T))^\sharp\|_{\mH}^q\bigr]=0,\\
J_{4,q} &:=& \mE\bigl[\|\int_0^T \nabla^{(2)}Z^{(2)}(X(t,x))X^\sharp(t,x)\alpha(t)dt\|_{\mH}^q\bigr] \notag\\
        &\le& C |h|^q
\Lambda^q(T)\bigl( T^{3q/2}\varepsilon^{(3-\alpha)q/2} + T^{q+1}\varepsilon^{3q/2-\alpha} \bigr), \label{7}\\
J_{5,q} &:=& \mE\bigl[\|\int_0^T \nabla Z^{(2)}(X(t,x))\alpha^\sharp(t)dt\|_{\mH}^q\bigr] \notag\\
        &\le& C|h|^q \Bigl[
\frac{1}{\xi^{q}(T)}
\bigl[ (T^{7q/2}+T^{9q/2})\varepsilon^{(3-\alpha)q/2}
      + (T^{3q+1}+T^{4q+1})\varepsilon^{3q/2-\alpha} \bigr] \no\\
&& + \frac{T^{3q}}{\xi^{2q}(T)}
\bigl[ (T^{3q/2}+T^{5q/2})\varepsilon^{(3-\alpha)q/2}
      + (T^{q+1}+T^{2q+1})\varepsilon^{3q/2-\alpha} \bigr] \no\\
&& + \frac{T^{2q}}{\bigl(\int_0^T \xi^2(s)ds\bigr)^q}
\bigl[
\bigl( \int_0^T \xi(s)\,ds \bigr)^q
\bigl( T^{3q/2}\varepsilon^{(3-\alpha)q/2}
      + T^{q+1}\varepsilon^{3q/2-\alpha} \bigr) \no\\
&&
+ T^q
\bigl[ (T^{3q/2}+T^{5q/2})\varepsilon^{(3-\alpha)q/2}
      + (T^{q+1}+T^{2q+1})\varepsilon^{3q/2-\alpha} \bigr]
\bigr] \no\\
&& + T^{5q/2}\varepsilon^{(3-\alpha)q/2}
      + T^{7q/2}\varepsilon^{(3-\alpha)q/2}
      + T^{2q+1}\varepsilon^{3q/2-\alpha}
      + T^{3q+1}\varepsilon^{3q/2-\alpha}
\Bigr]. \label{8}
\ee

{\bf Step 3:}

We apply H\"older's inequality together with (\ref{1})-(\ref{8}), and it follows that
\be
\mE\bigl[|M_1|^p\bigr] &\le& C|h|^p  \Lambda^p(T)  N_{2p}^{1/2}(A_{2p}^{(0)})^{1/2}, \label{9}\\
\mE\bigl[|M_2|^p\bigr] &\le& C|h|^p  \Lambda^p(T)  N_{4p}^{1/2}(A_{2p}^{(1)})^{1/2}, \label{10}\\
\mE\bigl[|M_3|^p\bigr] &\le& C  N_{2p}^{1/2} S_{4p}^{1/4} R_{4p}^{1/4}, \label{11}\\
\mE\bigl[|M_4|^p\bigr] &\le& C  N_{2p}^{1/2} S_{4p}^{1/4} J_{4,4p}^{1/4}, \label{12}\\
\mE\bigl[|M_5|^p\bigr] &\le& C  N_{2p}^{1/2} S_{4p}^{1/4} J_{5,4p}^{1/4}. \label{13}
\ee

{\bf Step 4:}

We now set $\varepsilon = T^{1/\alpha}$ in the bounds (\ref{2})-(\ref{8}). This yields the following explicit powers of $T$:
\ce
N_{2p}^{1/2} &\le& C T^{-3p/\alpha}, \notag\\
N_{4p}^{1/2} &\le& C T^{-6p/\alpha}, \notag\\
(A_{2p}^{(0)})^{1/2} &\le& C T^{2p/\alpha}, \notag\\
(A_{2p}^{(1)})^{1/2} &\le& C T^{5p/\alpha}, \notag\\
S_{4p}^{1/4} &\le& C T^{3p/(2\alpha)}, \notag\\
R_{4p}^{1/4} &=& 0, \notag\\
J_{4,4p}^{1/4} &\le& C|h|^p \Lambda^p(T) T^{p+3p/(2\alpha)}, \notag\\
J_{5,4p}^{1/4} &\le& C|h|^p\Bigl[
\frac{T^{3p+3p/(2\alpha)}+T^{4p+3p/(2\alpha)}}{\xi^p(T)}\\
&&+\frac{T^{4p+3p/(2\alpha)}+T^{5p+3p/(2\alpha)}}{\xi^{2p}(T)}\\
&&+\frac{\bigl(\int_0^T\xi(t)\,dt\bigr)^p T^{3p+3p/(2\alpha)}
      +T^{4p+3p/(2\alpha)}+T^{5p+3p/(2\alpha)}}{\bigl(\int_0^T\xi^2(t)\,dt\bigr)^p}\\
&&+T^{2p+3p/(2\alpha)}+T^{3p+3p/(2\alpha)}
\Bigr]. \notag
\de
Inserting these into (\ref{9})-(\ref{13}) and summing over the five terms via the triangle inequality, we obtain
\ce\label{14}
&&\mE\bigl[|M(T,x_1,x_2,h)|^p\bigr]\\
&\le& C|h|^p \Bigl[
\underbrace{\Lambda^p(T)T^{-p/\alpha}}_{\text{from }M_1,M_2}\\
&&+ \underbrace{\Lambda^p(T)}_{\text{from }M_4}\\
&&+ \underbrace{\frac{T^{4p}}{\xi^{p}(T)} + \frac{T^{5p}}{\xi^{2p}(T)}}_{\text{from }M_5}\\
&&+ \underbrace{\frac{T^{6p}}{\bigl(\int_0^T\xi^2(t)\,dt\bigr)^p}}_{\text{from }M_5}\\
&&+ \underbrace{\frac{\bigl(\int_0^T\xi(t)\,dt\bigr)^p T^{4p}}{\bigl(\int_0^T\xi^2(t)\,dt\bigr)^p}}_{\text{from }M_5}\\
&&+ \underbrace{T^{2p} + T^{3p}}_{\text{from }M_5}
\Bigr],
\de
where the constant $C$ depends only on $p$, $\alpha$, the bounds on the coefficients, and the norms of $\sigma$ and $\sigma^{-1}$, but not on $T$ or $h$. Then the desired estimate is obtained.

(2) We are now in a position to give the explicit expression for $\delta_{\sharp}(\tau^{\alpha}(T))$. Let
\ce
\gamma_1(t)&=&(\int_0^t\int_{\cO}\zeta(u)N(ds,du))^{-1},\\
\gamma_2(t)&=&\int_0^t\int_{\cO}\int_R\zeta^{1/2}(u)\xi^T(r)N\odot\rho(ds,du,dr),\\
\gamma_3(t)&=&\sigma^{-1}(h_2-\alpha^{(2)}(t)+\int_0^t\nabla Z^{(2)}(X(s,x))\alpha(s)ds).
\de
Then it is obvious that
\ce
\tau^{\alpha}(T)=\gamma_1(T)\gamma_2(T)\gamma_3(T).
\de
It is also worth noting that for $F\in\mD$, $G\in\text{dom}(\delta_{\sharp})$ such that $\delta_{\sharp}(G)F-\<G,F^{\sharp}\>_{\mH}\in L^2(\mP)$,
\ce
\delta_{\sharp}(FG)=\delta_{\sharp}(G)F-\<G,F^{\sharp}\>_{\mH}.
\de
It follows that
\ce
\delta_{\sharp}(\tau^{\alpha}(T))=\gamma_1(T)\delta_{\sharp}(\gamma_2(T))\gamma_3(T)+\<\gamma_1^{\sharp}(T),\gamma_2(T)\gamma_3(T)\>_{\mH}+\<\gamma_1(T)\gamma_2(T),\gamma_3^{\sharp}(T)\>_{\mH}.
\de
It can be seen from Lemma \ref{divergence particular} that
\ce
&&\delta_{\sharp}(\gamma_2(T))\\
&=&-\sum_{i=1}^d\int_0^T\int_{\cO}\bigl[\zeta^{1/2}(u)(\partial_i\log k(u)\zeta^{1/2}(u)+\frac{1}{2}\zeta^{-1/2}(u)\partial_i\zeta(u))\\
&&+\zeta^{1/2}(u)\partial_i\zeta^{1/2}(u)\bigr]e_i^TN(dt,du).
\de
It is also obvious by \cite[Proposition 8.2]{BD} that
\ce
&&\<\gamma_1^{\sharp}(T),\gamma_2(T)\gamma_3(T)\>_{\mH}\\
&=&(\int_0^T\int_{\cO}\zeta(u)N(dt,du))^{-2}\<\int_0^T\int_{\cO}\int_R\zeta^{1/2}(u)\sum_{i=1}^d\partial_i\zeta(u)\xi^i(r)N\odot\rho(dt,du,dr),\\
&&\int_0^T\int_{\cO}\int_R\zeta^{1/2}(u)\xi^T(r)N\odot\rho(dt,du,dr)\sigma^{-1}(h_2-\alpha^{(2)}(T)+\int_0^T\nabla Z^{(2)}(X(t,x))\alpha(t)dt)\>_{\mH}\\
&=&(\int_0^T\int_{\cO}\zeta(u)N(dt,du))^{-2}\int_0^T\int_{\cO}\zeta(u)(\partial_1\zeta(u),\cdots,\partial_d\zeta(u))N(dt,du)\\
&&\sigma^{-1}(h_2-\alpha^{(2)}(T)+\int_0^T\nabla Z^{(2)}(X(t,x))\alpha(t)dt)
\de
and
\ce
&&\<\gamma_1\gamma_2(T),\gamma_3^{\sharp}(T)\>_{\mH}\\
&=&(\int_0^T\int_{\cO}\zeta(u)N(dt,du))^{-1}\<\int_0^T\int_{\cO}\int_R\zeta^{1/2}(u)\xi^T(r)N\odot\rho(dt,du,dr),\\
&&\sigma^{-1}(-(\alpha^{(2)}(T))^{\sharp}+\int_0^T\nabla^{(2)}Z^{(2)}(X(t,x))X^{\sharp}(t,x)\alpha(t)dt+\int_0^T\nabla Z^{(2)}(X(t,x))\alpha^{\sharp}(t)dt)\>_{\mH}\\
&=&-(\int_0^T\int_{\cO}\zeta(u)N(dt,du))^{-1}\<\int_0^T\int_{\cO}\int_R\zeta^{1/2}(u)\xi^T(r)N\odot\rho(dt,du,dr),\\
&&\sigma^{-1}(\alpha^{(2)}(T))^{\sharp}\>_{\mH}\\
&&+(\int_0^T\int_{\cO}\zeta(u)N(dt,du))^{-1}\<\int_0^T\int_{\cO}\int_R\zeta^{1/2}(u)\xi^T(r)N\odot\rho(dt,du,dr),\\
&&\sigma^{-1}\int_0^T\nabla^{(2)}Z^{(2)}(X(t,x))X^{\sharp}(t,x)\alpha(t)dt\>_{\mH}\\
&&+(\int_0^T\int_{\cO}\zeta(u)N(dt,du))^{-1}\<\int_0^T\int_{\cO}\int_R\zeta^{1/2}(u)\xi^T(r)N\odot\rho(dt,du,dr),\\
&&\sigma^{-1}\int_0^T\nabla Z^{(2)}(X(t,x))\alpha^{\sharp}(t)dt\>_{\mH},
\de
where we have used the following fact:
\ce
\int_R\xi^i(r)\rho(dr)=0,\quad \int_R\xi^i(r)\xi^j(r)\rho(dr)=\delta_{ij}.
\de
Then the expression of $\delta_{\sharp}(\tau^{\alpha}(T))$ can be obtained.

\end{proof}

\subsection{Completion of the proof}

First of all, using \cite[Proposition 8.2]{BD}, we have
\ce
\<X^{\sharp}(t,x),\tau^{\alpha}(t)\>_{\mH}&=&\int_0^t\nabla Z(X(s,x))\<X^{\sharp}(s,x),\tau^{\alpha}(s)\>_{\mH}ds\\
&&+\<(0,\int_0^t\int_{\cO}\int_R\zeta^{1/2}(u)\sigma\xi(r)N\odot\rho(ds,du,dr)),\tau^{\alpha}(t)\>_{\mH}.
\de
On the other hand, using (\ref{alpha 1 variation}) and (\ref{tau formula variation}), it follows that
\ce
\int_0^t\nabla Z^{(1)}(X(s,x))\alpha(s)ds=\alpha^{(1)}(t)-h_1
\de
and
\ce
&&\<\int_0^t\int_{\cO}\int_R\zeta^{1/2}(u)\sigma\xi(r)N\odot\rho(ds,du,dr),\tau^{\alpha}(t)\>_{\mH}-\int_0^t\nabla Z^{(2)}(X(s,x))\alpha(s)ds\\
&=&h_2-\alpha^{(2)}(t).
\de
Now let
\ce
v(t):=\alpha(t)+\<X^{\sharp}(t,x),\tau^{\alpha}(t)\>_{\mH}.
\de
Then we have
\ce
v(t)&=&\alpha(t)+\int_0^t\nabla Z(X(s,x))\<X^{\sharp}(s,x),\tau^{\alpha}(s)\>_{\mH}ds\\
&&+\<(0,\int_0^t\int_{\cO}\int_R\zeta^{1/2}(u)\sigma\xi(r)N\odot\rho(ds,du,dr)),\tau^{\alpha}(t)\>_{\mH}\\
&=&\alpha(t)+\int_0^t\nabla Z(X(s,x))v(s)ds\\
&&+\<(0,\int_0^t\int_{\cO}\int_R\zeta^{1/2}(u)\sigma\xi(r)N\odot\rho(ds,du,dr)),\tau^{\alpha}(t)\>_{\mH}\\
&&-\int_0^t\nabla Z(X(s,x))\alpha(s)ds\\
&=&\alpha(t)+\int_0^t\nabla Z(X(s,x))v(s)ds\\
&&+\<(0,\int_0^t\int_{\cO}\int_R\zeta^{1/2}(u)\sigma\xi(r)N\odot\rho(ds,du,dr)),\tau^{\alpha}(t)\>_{\mH}-\int_0^t\nabla Z^{(2)}(X(s,x))\alpha(s)ds\\
&&-(\int_0^t\nabla Z^{(1)}(X(s,x))\alpha(s)ds,0)\\
&=&h+\int_0^t\nabla Z(X(s,x))v(s)ds.
\de
It is worth noting that the derivative process $(\nabla_hX(t,x),t\geq0)$ solves the following SDEs with jumps:
\ce
\nabla_hX(t,x)=h+\int_0^t\nabla Z(X(s,x))\nabla_hX(s,x)ds.
\de
This implies by $\alpha(T)=0$ that
\ce
\nabla_hX(T,x)=\<X^{\sharp}(T,x),\tau^{\alpha}(T)\>_{\mH}.
\de
Then for any $\varphi\in\cC_b^1(\mR^m\times\mR^d)$, it is obvious that
\ce
\nabla_hP_T\varphi(x)&=&\mE\bigl[\<\nabla\varphi(X(T,x)),\nabla_hX(T,x)\>\bigr]\\
&=&\mE\bigl[\<\nabla\varphi(X(T,x)),\<X^{\sharp}(T),\tau^{\alpha}(T)\>_{\mH}\>\bigr]\\
&=&\mE\bigl[\<(\varphi(X(T,x)))^{\sharp},\tau^{\alpha}(T)\>_{\mH}\bigr]\\
&=&\mE\bigl[\varphi(X(T,x))\delta_{\sharp}(\tau^{\alpha}(T))\bigr].
\de
Using Lemma \ref{divergence expression}, we complete the proof.


\begin{thebibliography}{999}
\bibitem{A}H. E. Altman: Bismut-Elworthy-Li formulae for Bessel processes, {\it S\'eminaire de Probabilit\'es XLIX}, 183--220, Lecture Notes in Math., 2215, Springer, Cham, 2018.
\bibitem{BFW}J. Bao, R. Fang and J. Wang: Exponential ergodicity of L\'{e}vy-driven Langevin dynamics with singular potentials, {\it Stochastic Process. Appl.}, 172 (2024), 104341.
\bibitem{BW}J. Bao and J. Wang: Coupling approach for exponential ergodicity of stochastic Hamiltonian systems with L\'{e}vy noises, {\it Stochastic Process. Appl.}, 146 (2022), 114--142.
\bibitem{BW2}J. Bao and J. Wang: $L_2$-exponential ergodicity of stochastic Hamiltonian systems with $\alpha$-stable L\'{e}vy noises, {\it Forum Math.}, 38 (2026), 1--25.
\bibitem{BD}N. Bouleau and L. Denis: {\it Dirichlet Forms Methods for Poisson Point Measures and L\'{e}vy Processes}, Springer, 2015.
\bibitem{BF}N. Bouleau and F. Hirsch: {\it Dirichlet Forms and Analysis on Wiener Space}, De Gruyter, 1991.
\bibitem{CK}Z. Chen and T. Kumagai: Heat kernel estimates for stable-like processes on $d$-sets, {\it Stochastic Process. Appl.}, 108, 27--62, 2003.
\bibitem{DM}C. Dellacherie and P. A. Meyer: {\it Probability and Potential}, North-Holland Mathematics Studies, 29, North-Holland Publishing Co., Amsterdam-New York, 1978.
\bibitem{DS}C. Deng and R. L. Schilling: On shift Harnack inequality for subordinate semigroups and moment estimate for L\'{e}vy processes, {\it Stochastic Process. Appl.}, 125, 3851--3878, 2015.
\bibitem{K}H. Kunita: Stochastic differential equations based on L\'{e}vy processes and stochastic flows of diffeomorphisms, in: Real and Stochastic Analysis, Trends Math. Birkh\"auser Boston, Boston, MA, pp. 305--373, 2004.
\bibitem{LWZ}Y. Liu, J. Wang and M. Zhang: Exponential contraction rates for a class of degenerate SDEs with L\'{e}vy noises, {\it J. Differential Equations}, 413 (2024), 1--33.
\bibitem{RZ}J. Ren and H. Zhang: Derivative formulae for stochastic differential equations driven by Poisson random measures, {\it J. Math. Anal. Appl.}, 462, 554--576, 2018.
\bibitem{RW}P. Ren and F. Wang: Bismut formula for Lions derivative of distribution dependent SDEs and applications, {\it J. Differential Equations}, 267 (8), 4745--4777, 2019.
\bibitem{SZ}Y. Song and X. Zhang: Regularity of density for SDEs driven by degenerate L\'{e}vy noises, {\it Electron. J. Probab.}, 20 (21), 1--27, 2015.
\bibitem{T}A. Takeuchi: The Bismut-Elworthy-Li type formulae for stochastic differential equations with jumps, {\it J. Theoret. Probab.}, 23, 576--604, 2010.
\bibitem{WXZ1}L. Wang, L. Xie and X. Zhang: Derivative formulae for SDEs driven by multiplicative $\alpha$-stable-like processes, {\it Stochastic Process. Appl.}, 125, 867--885, 2015.
\bibitem{W1}F. Wang: Derivative formula and gradient estimates for Gruschin type semigroups, {\it J. Theoret. Probab.}, 27, 80--95, 2014.
\bibitem{W2}F. Wang: Derivative formulas and Poincar\'e inequality for Kohn-Laplacian type semigroups, {\it Sci. China Math.}, 59, 261--280, 2016.
\bibitem{WZ}F. Wang and X. Zhang: Derivative formula and applications for degenerate diffusion semigroups, {\it J. Math. Pures Appl.}, 99, 726--740, 2013.
\bibitem{WXZ2}F. Wang, L. Xu and X. Zhang: Gradient estimates for SDEs driven by multiplicative L\'{e}vy noise, {\it J. Funct. Anal.}, 269 (10), 3195--3219, 2013.
\bibitem{Zhang2026a}H. Zhang: Derivative formulae and gradient estimates for nonlocal Gruschin-type semigroups, Preprint.
\bibitem{Zhang2026b}H. Zhang: Bismut-type and Driver-type derivative formulae for nonlocal Kohn-Laplacian type semigroups, Preprint.
\bibitem{Z1}X. Zhang: Derivative formula and gradient estimate for SDEs driven by $\alpha$-stable processes, {\it Stochastic Process. Appl.}, 123, 1213--1228, 2013.
\end{thebibliography}
\end{document}